\documentclass[a4paper, twoside, 11pt]{scrartcl}
\pdfoutput=1

\usepackage[a4paper, top=88pt, bottom=88pt, left=72pt, right=72pt, headsep=16pt, footskip=28pt]{geometry}
\usepackage{amsfonts,amssymb,amsmath,amsthm,amstext,amssymb,amsopn,mathtools,nicefrac,xfrac}
\usepackage[authoryear,sort,round]{natbib} 
\usepackage[utf8]{inputenc}                 
\usepackage{booktabs}                       
\usepackage{nicefrac}                       
\usepackage{microtype}                      
\usepackage{graphics,graphicx}              
\usepackage{subcaption}                     
\usepackage[ruled,linesnumbered]{algorithm2e} 
\usepackage{verbatim}                       
\usepackage{appendix}                       
\usepackage{bm}                             
\usepackage{array}                          
\newcolumntype{C}[1]{>{\centering\arraybackslash}m{#1}}  
\usepackage[usenames,dvipsnames]{xcolor}    
\usepackage{hyperref}                       
\hypersetup{colorlinks=true,linkcolor=blue,filecolor=magenta,urlcolor=blue,citecolor=OliveGreen}
\usepackage[noabbrev,capitalise,nosort,nameinlink]{cleveref}  
\usepackage{bbm}                            
\usepackage{mathtools}                      

\newcommand*{\arXiv}[1]{\bgroup\color{blue}\href{https://arxiv.org/abs/#1}{arXiv:#1}\egroup}
\newcommand*{\doi}[1]{\bgroup\color{blue}\href{https://doi.org/#1}{doi:#1}\egroup}
\newcommand*{\email}[1]{\bgroup\color{blue}\href{mailto:#1}{#1}\egroup}
\renewcommand*{\url}[1]{\bgroup\color{blue}\href{#1}{#1}\egroup}
\usepackage{enumitem, moreenum}
\setlist[enumerate]{nosep}
\setlist[itemize]{nosep}
\usepackage{mleftright} \mleftright
\renewcommand{\qedsymbol}{$\blacksquare$}
\renewenvironment{proof}[1][\proofname]{\noindent{\bfseries\sffamily #1.} }{\hfill\qedsymbol\medskip}
\usepackage[labelfont={sf,bf}]{caption}
\DeclareCaptionLabelSeparator{figlabelsep}{\,\,\,}
\usepackage{scrlayer-scrpage, xhfill}
\automark[section]{section}
\setkomafont{pageheadfoot}{\normalcolor\sffamily}
\setkomafont{pagenumber}{\normalfont\normalsize\sffamily}
\clearpairofpagestyles
\let\oldtitle\title
\renewcommand{\title}[1]{\oldtitle{#1}\newcommand{\theshorttitle}{#1}}
\newcommand{\shorttitle}[1]{\renewcommand{\theshorttitle}{#1}}
\let\oldauthor\author
\renewcommand{\author}[1]{\oldauthor{#1}\newcommand{\theshortauthor}{#1}}
\newcommand{\shortauthor}[1]{\renewcommand{\theshortauthor}{#1}}
\cohead{\xrfill[0.525ex]{0.6pt}~\theshorttitle~\xrfill[0.525ex]{0.6pt}}
\cehead{\xrfill[0.525ex]{0.6pt}~\theshortauthor~\xrfill[0.525ex]{0.6pt}}
\newcommand{\theabstract}[1]{\par\bgroup\noindent\textbf{\textsf{Abstract.}} #1\egroup}
\newcommand{\thekeywords}[1]{\par\smallskip\bgroup\noindent\textbf{\textsf{Keywords.}}\newcommand{\and}{ $\bullet$ } #1\egroup}
\newcommand{\themsc}[1]{\par\smallskip\bgroup\noindent\textbf{\textsf{2020 Mathematics Subject Classification.}}\newcommand{\and}{ $\bullet$ } #1\egroup}
\newcommand*{\affilref}[1]{\ref{affiliation#1}}
\newcommand*{\affiliation}[3]{
	\footnotetext[#1]{\label{affiliation#2} #3}
}

\numberwithin{equation}{section}
\numberwithin{figure}{section}
\numberwithin{table}{section}

\renewcommand*{\geq}{\geqslant}

\renewcommand*{\leq}{\leqslant}
\newcommand*{\rd}{\mathrm{d}}
\newcommand*{\Reals}{\mathbb{R}}

\newcommand*{\F}{\mathcal{F}_{\Delta T}}
\newcommand*{\G}{\mathcal{G}_{\Delta T}}
\newcommand*{\E}{\mathbb{E}}
\newcommand*{\bu}{\bm{u}}
\newcommand*{\bv}{\bm{v}}
\newcommand*{\bU}{\bm{U}}
\newcommand*{\bXi}{\bm{\xi}}
\newcommand*{\LF}{L_{\mathcal{F}}}
\newcommand*{\LG}{L_{\mathcal{G}}}

\newtheorem{theorem}{\sffamily Theorem}[section]
\newtheorem{corollary}[theorem]{\sffamily Corollary}
\newtheorem{lemma}[theorem]{\sffamily Lemma}

\theoremstyle{definition}
\newtheorem{definition}[theorem]{\sffamily Definition}

\newtheorem{assumption}[theorem]{\sffamily Assumption}
\newtheorem{remark}[theorem]{\sffamily Remark}

\renewenvironment{proof}[1][\proofname]{\noindent{\bfseries\sffamily #1.} }{\hfill\qedsymbol\medskip}

\crefname{assumption}{Assumption}{Assumptions}
\Crefname{assumption}{Assumption}{Assumptions}

\title{Error bound analysis of the stochastic parareal algorithm}
\shorttitle{Stochastic parareal analysis}
\author{
    Kamran Pentland\textsuperscript{\affilref{Warwick}} 
    \and
    Massimiliano Tamborrino\textsuperscript{\affilref{Warwick2}} 
    \and
    T.~J.~Sullivan\textsuperscript{\affilref{Warwick},\affilref{Turing}}
}
\shortauthor{K.~Pentland, M.~Tamborrino, and T.~J.~Sullivan}
\date{\today}

\begin{document}
\maketitle

\affiliation{1}{Warwick}{Department of Mathematics, University of Warwick, Coventry, CV4 7AL, United Kingdom\newline (\email{kamran.pentland@warwick.ac.uk},  \email{t.j.sullivan@warwick.ac.uk})}
\affiliation{2}{Warwick2}{Department of Statistics, University of Warwick, Coventry, CV4 7AL, United Kingdom\newline (\email{massimiliano.tamborrino@warwick.ac.uk})}
\affiliation{3}{Turing}{Alan Turing Institute, British Library, 96 Euston Road, London NW1 2DB, United Kingdom}


\begin{abstract}\small
    \theabstract{%
    Stochastic parareal (SParareal) is a probabilistic variant of the popular parallel-in-time algorithm known as parareal.
    Similarly to parareal, it combines fine- and coarse-grained solutions to an ordinary differential equation (ODE) using a predictor-corrector (PC) scheme.
    The key difference is that carefully chosen random perturbations are added to the PC to try to accelerate the location of a stochastic solution to the ODE.
    In this paper, we derive superlinear and linear mean-square error bounds for SParareal applied to nonlinear systems of ODEs using different types of perturbations.
    We illustrate these bounds numerically on a linear system of ODEs and a scalar nonlinear ODE, showing a good match between theory and numerics.
    }
    \thekeywords{%
        {Parallel-in-time}%
        \and%
        {stochastic parareal}%
        \and%
        {error bounds}%
        \and%
        {ordinary differential equations}%
    }
    \themsc{
        {65L70}
        \and%
        {65Y05}
        \and%
        {65C99}
    }
\end{abstract}



\section{Introduction} \label{sec:intro} 

In recent years, there has been an increasing need to develop faster numerical integration methods for a broad range of time-dependent ordinary, partial, and stochastic differential equations (ODEs/PDEs/SDEs) that form the building blocks of mathematical models of real-world systems.
Advances in high performance computing have fuelled the development of new parallel algorithms for solving such initial value problems (IVPs), with one particular area of increasing focus being parallel-in-time (PinT) methods\footnote{A collection of PinT research papers and software can found at \url{https://parallel-in-time.org/}.}.
Given the causal nature of time, PinT algorithms provide a way to, either iteratively or directly, simulate the solution states of an IVP across the entire time interval of integration concurrently.
Since the seminal work of \cite{Nievergelt1964}, a variety of different PinT methods have been developed over the last $60$ years---refer to \cite{gander2015} and \cite{ong2020} for reviews.

Of particular interest here is the increasingly popular multiple shooting-type/multigrid PinT algorithm know as parareal \citep{lions2001}.
Dividing the time interval into $N$ `slices', typically of fixed size $\Delta T$, parareal iteratively combines solutions calculated by two serial numerical integrators (one coarse- and one fine-grained) on each slice using a predictor-corrector (PC) scheme.
Using $N$ processors, one for each time slice, parareal runs the fine (expensive) computations in parallel, locating a solution to the IVP in a fixed number of iterations $k \leq N$, after which
one says the algorithm has `converged'.
Parallel speedup from parareal is bounded above by $N/k$ and so, ideally, the algorithm finds a solution in $k \ll N$ iterations. 
Since it was proposed, parareal has been shown to provide speedup for a range of different IVPs in areas ranging from molecular dynamics \citep{engblom2009,legoll2020,legoll2022} to plasma physics \citep{samaddar2010,samaddar2019,grigori2021}---refer to \cite[Sec. 2]{ong2020} for an overview.
Variants have also been developed to tackle issues with Hamiltonian systems \citep{bal2008,dai2013}, task scheduling \citep{elwasif2011}, and IVP stiffness \citep{maday2020}.
Others, aimed at reducing the number iterations in parareal, have also emerged.
Approaches include learning the correction term in the PC using Gaussian process emulators \citep{pentland2022b} and building the coarse solver using a Krylov subspace of the set of PC solutions \citep{gander2008krylov}.

The focus of this paper is the \textit{stochastic parareal} algorithm (henceforth SParareal), developed by \cite{pentland2022a}, in which random perturbations are added to the PC to increase the probability of locating the solution to the IVP in fewer iterations.
In each time slice, $M$ samples are taken from probability distributions constructed using different combinations of coarse and fine solution information (known as ``sampling rules'') and propagated forward in time in parallel.
From this set of perturbed solutions, the ones which generate the most continuous (fine) trajectory in solution space, starting from the exact known initial condition, are selected.
Numerical simulations showed that as the number of samples increased, the probability of locating the exact solution increased and so convergence would occur in fewer iterations and therefore faster wallclock time.
Numerical simulations also showed that, upon multiple simulations of SParareal, one could obtain a distribution of solutions to the IVP which were numerically accurate with respect to the serially located fine solution.
\cite{conrad2017} and \cite{lie2019,lie2022} developed similar (non-time-parallel) numerical schemes in which they perturb the solution states generated by serial numerical integrators to incorporate a measure of uncertainty quantification in the solution of ODEs.

\subsection{Contribution and outline} \label{subsec:contribution}
Deriving rigorous error bounds for SParareal (and other PinT methods in general) is important in demonstrating that numerical solutions obtained in parallel are meaningful, accurate, and that they can be compared to one another \citep{gander2022}.
The first convergence results for parareal assumed that $k$ was fixed, showing that the error of the scheme approached zero as the time slice size $\Delta T \rightarrow 0$ \citep{lions2001,bal2002,bal2005}.
Here, we are interested in studying how numerical errors behave when $\Delta T$ is fixed and the number of iterations $k$ grows.
\cite{gander2007} investigated this, deriving both linear and superlinear error bounds for the scalar linear ODE problem on unbounded and bounded time intervals, respectively. 
Following this, \cite{gander2008} used the generating function method to derive a superlinear bound for nonlinear systems of ODEs on bounded intervals.
Whereas work has been done to derive error bounds for parareal applied to certain SDEs \citep{bal2006,engblom2009}, where the IVP itself contains randomness, those results cannot be applied here, as it is the SParareal scheme itself that contains randomness, not the solvers or the IVP.

In this paper, we extend the qualitative discussion of numerical convergence in \cite{pentland2022a} by making use of convergence analysis from both the parareal literature (see above) and the sampling-based ODE solver presented in \cite{lie2019}.
We derive explicit mean-square error bounds for SParareal applied to nonlinear systems of ODEs (over a finite time interval), using both \textit{state-independent} and \textit{state-dependent} perturbations.
The rest of the paper is organised as follows.
In \cref{sec:parareal}, we set up the IVP and provide an overview of the (classic) parareal algorithm.
We then detail the SParareal scheme in \cref{sec:SParareal}, outlining how it works and defining the (state-dependent) sampling rules.
In \cref{sec:convergence}, we begin by outlining the assumptions on the fine and coarse integrators required to derive the error bounds. 
We first consider the \textit{state-independent} setting, where the random perturbations do not depend on the current time step or iteration and are assumed to have bounded absolute moments.
In this setting we derive our main result (\cref{thm:superlinear_bound}), a superlinear bound on the mean-square error depending on the time step and iteration. 
Using this result, we can maximise the error over time (fixing the iteration), to derive a linear error bound (\cref{corr:linear_bound}).
In the \textit{state-dependent} setting, we allow the perturbations to depend on the current state of the system, i.e. the coarse and fine solution information at the current time step and iteration, to analyse the convergence of the sampling rules proposed in the original work. 
We derive linear bounds (\cref{corr:SR_24,corr:SR_13}) in this case. 
Following this, we verify these bounds by applying them numerically to a linear system of ODEs and a nonlinear scalar ODE in \cref{sec:numerics}.
We conclude with some brief remarks on the theoretical and numerical results and discuss possible directions for future study in \cref{sec:conclusions}.
In the Appendix, we provide some technical results used to derive the aforementioned error bounds.

\subsection{Notation} \label{subsec:notation}
Let $\E$ denote the expected value of a random variable.
Variables $\bu, \bv \in \Reals^d$ will be $d$-dimensional real-valued vectors unless otherwise stated.
Denote the component-wise absolute value of a vector as $|\bu| = (|u_1|,\ldots,|u_d|)^{\intercal}$ and the Hadamard (component-wise) product as $\bu \circ \bv = (u_1 v_1,\ldots,u_d v_d)^{\intercal}$.
Also note that $\bu^2$ will correspond to component-wise squaring, i.e.\ $\bu^2 = \bu \circ \bu$.
We let $\| \bu \|$ denote the infinity (or uniform) norm, i.e.\ $\| \bu \| = \max_{i=1,\ldots,d}|u_i|$.
The $d$-dimensional vector of ones and the identity matrix will be written as $\bm{1}$ and $I_d$,  respectively.
Non-negative constants will be denoted throughout by $C_1$, $C_2$,\ldots.



\section{The parareal scheme} \label{sec:parareal}
In this section, we describe the IVP under consideration and define the classic parareal scheme.

\subsection{Problem setup}
Consider the following system of $d \in \mathbb{N}$ autonomous ODEs
\begin{equation} \label{eq:systemODE} 
    \frac{\rd \bu}{\rd t} = \bm{f} \bigl( \bu(t) \bigr) \quad \text{over} \quad t \in [t_0,T], \quad \text{with} \quad \bu(t_0) = \bu_0,
\end{equation}
where $\bm{f}\colon \Reals^d \to \Reals^d$ is a nonlinear vector field, $\bu\colon [t_0,T] \to \Reals^d$ is the time-dependent solution, and $\bu_0 \in \Reals^d$ is the initial value at time $t_0$.
We assume $\bm{f}$ is sufficiently smooth such that \eqref{eq:systemODE} has a unique solution for all initial conditions of interest.
We seek numerical solutions $\bU_n \approx \bu(t_n)$ to the IVP \eqref{eq:systemODE} on a pre-defined mesh $\bm{t} = (t_0,\dots,t_N)$, where $t_n = t_0 + n\Delta T$ for fixed $\Delta T = (T-t_0)/N$.
$N$ is the number of processors required for parareal to compute a solution in parallel, i.e.\ one processor is assigned to each time slice $[t_n,t_{n+1}]$, $n = 0,\dots,N-1$.
Note that everything that follows extends to the nonautonomous case but is not discussed here to simplify explanation and notation. 

\subsection{The scheme}
The parareal algorithm uses $N$ processors and two deterministic numerical flow maps (solvers) to locate solutions $\bU^k_n$, at iteration $k$ and time $t_n$, to system \eqref{eq:systemODE}. 
These maps take an initial state $\bU^k_n$ at time $t_n$ and propagate it, over a time slice of size $\Delta T$, to a terminal state $\bU^k_{n+1}$ at time $t_{n+1}$. 
The \textit{fine} solver, denoted as the flow map $\F\colon \Reals^d \rightarrow \Reals^d$, returns a terminal state with high numerical accuracy at very high computational cost.
The cost is high enough that trying to use $\F$ to solve \eqref{eq:systemODE} sequentially, i.e.\ calculate
\begin{align} \label{eq:exact_F_sol}
    \bU_{n+1} = \F (\bU_n) \quad \text{for} \quad n=0,\ldots,N-1,
\end{align}
is computationally infeasible. 
The \textit{coarse} solver, denoted similarly by $\G\colon \Reals^d \rightarrow \Reals^d$, returns a terminal state with a much lower numerical accuracy, at a smaller computational cost than $\F$.
Therefore, $\G$ is fast and allowed to run serially at any time, whilst $\F$ is expensive and must only be run in parallel.
\begin{figure}[t!]
    \centering
    \includegraphics[width=0.95\linewidth]{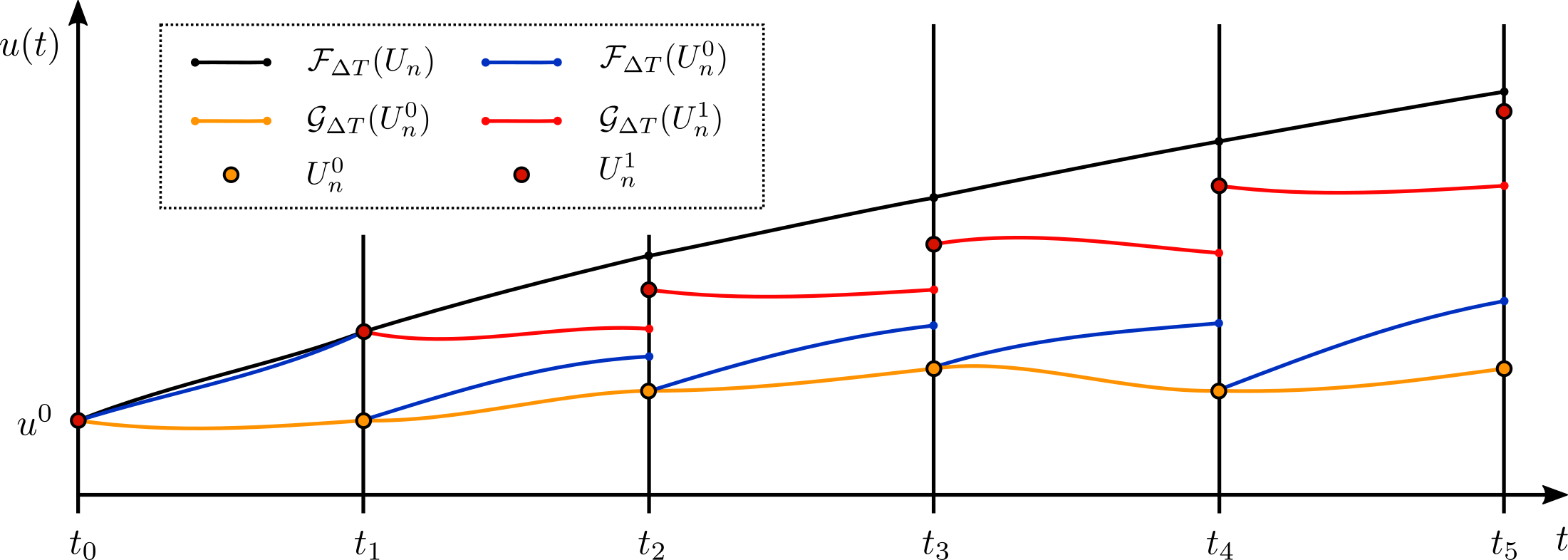}
    \caption{First iteration of parareal that numerically approximates the exact solution (black line), obtained via repeated serial applications of the fine solver, i.e.\ \eqref{eq:exact_F_sol} in the scalar case.
    The coarse initial guess found using $\G$ (yellow lines and dots) is followed by the parallel runs of $\F$ from these guesses (blue lines). 
    The coarse predictions from $\G$ (red lines) are then used in the prediction-correction update \eqref{eq:P3} (red dots).
    Figure adapted from \cite{pentland2022a} (Fig. 2).}
    \label{fig:parareal}
\end{figure}

\begin{definition}[Parareal] \label{def:parareal}
For the two numerical flow maps $\F$ and $\G$ described above, the parareal scheme is given by
\begin{align}
    \bU^0_0 &= \bu_0, \label{eq:P1} \\
    \bU^0_{n+1} &= \G(\bU^0_n), \quad && 0 \leq n \leq N-1, \label{eq:P2} \\
    \bU^{k+1}_{n+1} &= \G(\bU^{k+1}_n) + \F(\bU^k_n) - \G(\bU^k_n), \quad && 0 \leq k \leq n \leq N-1. \label{eq:P3}
\end{align}
\end{definition}
Parareal begins by propagating the exact initial condition \eqref{eq:P1} across $[t_0,t_N]$ serially using $\G$ \eqref{eq:P2}.
Each of the solution states found in \eqref{eq:P2} are then propagated in parallel using $\F$.
Equation \eqref{eq:P3}, often referred to as the PC, is then run serially to update the solution states at each $t_n$.
This process of propagating using $\F$ and updating using the PC can be repeated iteratively, stopping after $k$ iterations, once a pre-specified stopping criterion is met.
The stopping criteria can take many different forms, one popular choice being that the time slices up to $t_I$ are considered converged if
\begin{align} \label{eq:tolerance}
    \| \bU^k_n - \bU^{k-1}_n \| < \varepsilon \quad \forall n \leq I,
\end{align}
for some user-specified tolerance $\varepsilon > 0$.
Once $I=N$, parareal is said to have converged in $k$ iterations.
An illustration of the first iteration of parareal is given in \cref{fig:parareal}.

\begin{remark} \label{rem:parareal}
It is often assumed, and will be assumed in this paper, that $\F$, if it were computationally feasible to apply serially over $[t_0,T]$, returns the `exact' solution to \eqref{eq:systemODE}, i.e.\ it returns $\bU_{n+1} = \bu(t_{n+1}) = \F(\bu(t_n))$, $n=0,\ldots,N-1$.
Therefore, after $k$ iterations, parareal returns the exact solution $\bU^k_n = \bU_n$ for $k \geq n$, meaning that parareal returns the exact solution over $[t_0,T]$ in at most $N$ iterations---albeit achieving no parallel speedup in this case.
This is an important feature of parareal which is also preserved by SParareal. 
\end{remark}

\section{The stochastic parareal scheme} \label{sec:SParareal}
In this section, we outline the SParareal scheme, how it works, and how we have slightly modified the scheme, compared to the original formulation in \cite{pentland2022a}, to carry out the error bound analysis in \cref{sec:convergence}. 
Following this, we describe the sampling rules tested in the original work. 

\subsection{The scheme}
The intuition behind SParareal is to perturb the solution states $\bU^k_n$ in the classic parareal scheme, i.e.\ the PC \eqref{eq:P3}, with some additive noise to reduce the number of iterations $k$ taken until the stopping tolerance \eqref{eq:tolerance} is met.
Recall that a reduction in $k$ by even a single iteration can correspond to a large increase in numerical speedup---by approximately the runtime of a single run of $\F$.

First, let us formally define the scheme.
\begin{definition}[SParareal] \label{def:Sparareal}
For the two numerical flow maps, $\F$ and $\G$, described in \cref{sec:parareal}, the stochastic parareal scheme is given by
\begin{align} 
    \bU^0_0 &= \bu_0, \label{eq:SP1} \\
    \bU^0_{n+1} &= \G(\bU^0_n), \quad && 0 \leq n \leq N-1, \label{eq:SP2} \\
    \bU^{1}_{n+1} &= \G(\bU^{1}_n) + \F(\bU^0_n) - \G(\bU^0_n), \quad && 0 \leq n \leq N-1, \label{eq:SP3} \\
    \bU^{k+1}_{n+1} &= \G(\bU^{k+1}_n) + \F(\bU^k_n) - \G(\bU^k_n) + \bXi^k_n(\bU^k_n), \quad && 1 \leq k \leq n \leq N-1, \label{eq:SP4}
\end{align}
where $\bXi^k_n(\bU^k_n)$ are (possibly state-dependent) random variables. Note that $\bXi^k_n(\bU^k_n) \equiv \bm{0}$ when $n=k$.
\end{definition}

The first three stages of the scheme \eqref{eq:SP1}-\eqref{eq:SP3} are identical to the `zeroth' and first iteration of parareal. 
The exact initial condition \eqref{eq:SP1} is propagated forward in time using the coarse solver \eqref{eq:SP2}, then there is a first pass of the PC \eqref{eq:SP3}. 
Following this, the stochastic iterations begin \eqref{eq:SP4}, whereby random perturbations, i.e.\ a single draw from the random variable $\bXi^k_n(\bU^k_n)$, are added to the PC solution.
Notice that no random perturbation is added when $n=k$ to ensure that SParareal returns the exact solution up to time $t_k$ after $k$ iterations, just as parareal does, recall \cref{rem:parareal}.
The reason the random perturbations are only included from iteration $k\geq1$ onward is because $\bXi^k_n(\bU^k_n)$ may depend on solution information from iteration $k-1$.
In \cref{sec:samplingrules}, we will discuss the construction of these random variables in more detail.
Note that the parareal scheme (\cref{def:parareal}) can be recovered by setting $\bXi^k_n(\bU^k_n) \equiv \bm{0}$ for $n \geq k$ in \eqref{eq:SP4}.

The scheme in \cref{def:Sparareal} looks slightly different to the one presented in \cite{pentland2022a}, where 
random perturbations were instead incorporated via random variables, denoted by $\bm{\alpha}^k_n$, in the correction term, i.e.\ equation \eqref{eq:SP4} was instead written as 
\begin{align} \label{eq:oldSP}
    \bU^{k+1}_{n+1} &= \G(\bU^{k+1}_n) + \F(\bm{\alpha}^k_n) - \G(\bm{\alpha}^k_n), \quad 1 \leq k \leq n \leq N-1.
\end{align}
The different state-dependent forms that $\bm{\alpha}^k_n$ can take are defined through the ``sampling rules'' in \cref{sec:samplingrules}.
Also note that $\bm{\alpha}^k_n = \bU^k_n$ when $n = k$, which is equivalent to the condition that $\bXi^k_n(\bU^k_n) \equiv \bm{0}$ when $n=k$ in \cref{def:Sparareal}.
This version \ref{eq:oldSP} of the scheme was designed so that $M\geq1$ samples could be drawn from each of the random variables $\bm{\alpha}^k_n$ to increase the probability of locating the exact solution state $\bU_n$ in fewer iterations. 
From the sets of samples generated at each $t_n$, all having been propagated in parallel using $\F$ and $\G$ (recall \eqref{eq:oldSP}), those generating the most continuous $\F$ trajectory across $[t_0,t_N]$ were then chosen as the ``best'' perturbations $\bm{\hat{\alpha}}^k_n$.
Numerical experiments illustrated that increasing $M$ led to further and further reductions in $k$, albeit at the cost of requiring more processors, specifically $\mathcal{O}(MN)$ in SParareal vs $\mathcal{O}(N)$ in parareal.

To enable us to carry out the convergence analysis, we move the random perturbations in \eqref{eq:oldSP} outside the correction term, and express them using $\bXi^k_n(\bU^k_n)$ in \eqref{eq:SP4}.
This new scheme is equivalent to the old scheme in the case where one sample ($M=1$) is drawn at each time $t_n$.
To obtain the appropriate values for $\bXi^k_n(\bU^k_n))$, we equate \eqref{eq:SP4} and \eqref{eq:oldSP} to find
\begin{align} \label{eq:Xi}
    \bXi^k_n(\bU^k_n) = \big(\F(\bm{\alpha}^k_n) - \G(\bm{\alpha}^k_n)\big) - \big(\F(\bU^k_n) - \G(\bU^k_n)\big).
\end{align}
Using the scheme in \cref{def:Sparareal}, we can derive error bounds for state-independent perturbations, i.e.\ $\bXi^k_n(\bU^k_n) = \bXi^k_n$, that have bounded absolute moments.
Using this result, we can then derive corresponding bounds for the state-dependent sampling rules summarised in \cref{table:1} using relation \eqref{eq:Xi}.
All of these bounds are derived assuming one sample ($M=1$) is drawn from each $\bXi^k_n(\bU^k_n)$. The $M$ sample case is much more complex and out of the scope of the present work. 

\subsection{Sampling Rules} \label{sec:samplingrules}
The sampling rules presented by \cite{pentland2022a} describe the probability distributions that $\bm{\alpha}^k_n$ follow in the SParareal algorithm, see \cref{table:1}.
These distributions were designed to vary with both iteration $k$ and time step $n$, so that as the solution states $\bU^k_n$ get closer to $\bU_n$, their variances would decrease.
The benefit of this property will be highlighted in \cref{sec:numerics}.
The different rules were constructed to assess the performance of SParareal when the perturbations had different distribution families, marginal means, or correlations.

\begin{table}[tbhp]
{\footnotesize
\caption{Sampling rules that the random variables $\bm{\alpha}^k_n$ follow.
The quantities $\bm{z}^k_n \sim \mathcal{N}( \bm{0},I_d)$ and $\bm{w}^k_n \sim \mathcal{U}([0,1]^d)$ are $d$-dimensional Gaussian and uniform random vectors, respectively, whilst $\bm{\sigma}^k_n = | \G(\bU^k_{n-1}) - \G(\bU^{k-1}_{n-1}) |$.} \label{table:1}
\begin{center}
\begin{tabular}{|C{0.2\linewidth}|C{0.4\linewidth}|} \hline \hline
\textbf{Sampling rule} & $\bm{\alpha}^k_n$ \\ \hline
$1$  & $\F(\bU^{k-1}_{n-1}) + (\bm{\sigma}^k_n \circ \bm{z}^k_n)$ \\ \hline
$2$  & $\bU^k_n + (\bm{\sigma}^k_n \circ \bm{z}^k_n)$      \\ \hline
$3$  & $\F(\bU^{k-1}_{n-1}) + \big( \sqrt{3} \bm{\sigma}^k_n \circ ( 2 \bm{w}^k_n - \bm{1}) \big)$ \\ \hline
$4$  & $\bU^k_n + \big( \sqrt{3} \bm{\sigma}^k_n \circ ( 2 \bm{w}^k_n - \bm{1}) \big)$ \\
\hline \hline
\end{tabular}
\end{center}
}
\end{table}

Sampling rules 1 and 2 correspond to multivariate Gaussian perturbations with marginal means $\F(\bU^{k-1}_{n-1})$ and $\bU^k_n$, respectively, and marginal standard deviations $\bm{\sigma}^k_n = | \G(\bU^k_{n-1}) - \G(\bU^{k-1}_{n-1}) |$.
The variable $\bm{z}^k_n \sim \mathcal{N}( \bm{0},I_d)$ is a standard $d$-dimensional Gaussian random vector.
Sampling rules 3 and 4 correspond to perturbations following a multivariate uniform distribution with the same marginal means and standard deviations as rules 1 and 2, respectively.
The variable $\bm{w}^k_n \sim \mathcal{U}([0,1]^d)$ is a standard $d$-dimensional uniformly distributed random vector with independent components.
Note that \cite{pentland2022a} considered both correlated and uncorrelated random variables $\bm{\alpha}^k_n$ in their experiments, whereas here we carry out our analysis only for the uncorrelated case for simplicity.

\section{Error bound analysis} \label{sec:convergence}
In this section, we analyse the mean-square error
\begin{align} \label{eq:MSE}
    e^k_n := \E \big[\| \bu(t_n) - \bU^k_n \|^2 \big]
\end{align}
between the exact solutions $\bu(t_n)$ and the stochastic numerical solutions $\bU^k_n$ located by the SParareal scheme.
We also define the maximal mean-square error (over time) at iteration $k$ to be
\begin{align} \label{eq:MSE2}
    \hat{e}^k := \max\limits_{1 \leq n \leq N} \{e^k_n\}.
\end{align}
Specifically, we analyse the mean-square error $e^k_n$ for the nonlinear autonomous system of ODEs in \eqref{eq:systemODE}, first deriving superlinear (\cref{thm:superlinear_bound}) and linear (\cref{corr:linear_bound}) bounds using state-independent perturbations in SParareal.
Then, using these results, we derive linear bounds for the state-dependent sampling rules 2 and 4 (\cref{corr:SR_24}) and 1 and 3 (\cref{corr:SR_13}).
In the following, we introduce some assumptions on the flow maps \citep{gander2008} and perturbations \citep{lie2019} required to derive the error bounds.

\begin{assumption}[Exact flow map] \label{assump:exactflow}
The flow map $\F$ solves \eqref{eq:systemODE} exactly such that
\begin{align} \label{eq:exact}
    \bu(t_{n+1}) = \F (\bu(t_n)).
\end{align}
\end{assumption}
\noindent This assumption is made for simplicity, since SParareal is essentially trying to locate the solution that would be obtained by running the fine solver serially, i.e.\ \eqref{eq:exact_F_sol}, in parallel.
If instead we were to consider $\F$ to be a numerical method with some (very small) numerical error with respect to the exact solution, then the accuracy of $\F$ would provide a lower bound on the accuracy of the SParareal scheme as a whole.
\begin{assumption}[Coarse flow map] \label{assump:coarseapprox}
The flow map $\G$ is a one-step numerical method with uniform local truncation error $\mathcal{O}(\Delta T^{p+1})$, for $p \geq 1$, such that
\begin{align} \label{eq:FG}
    \F(\bu) - \G(\bu) = c_1 (\bu) \Delta T^{p+1} + c_2 (\bu) \Delta T^{p+2} + \ldots,
\end{align}
for $\bu \in \Reals^d$ and continuously differentiable functions $c_i(\bu)$. 
Taking the difference of \eqref{eq:FG} evaluated at states $\bu, \bv \in \Reals^d$, then applying norms and the triangle inequality, we can write
\begin{align} \label{eq:FG2}
    \| \left(\F(\bu) - \G(\bu)\right) - \left(\F(\bv) - \G(\bv)\right) \| \leq C_1 \Delta T^{p+1} \| \bu - \bv \|,
\end{align}
where $C_1>0$ is the Lipschitz constant for $c_1$ and we absorb terms $\mathcal{O}(\Delta T^{p+2})$ into $C_1$.
\end{assumption}

\begin{assumption}[Lipschitz coarse flow] \label{assump:lipschitz}
The flow map $\G$ satisfies the Lipschitz condition
\begin{align} \label{eq:lipschitz}
     \| \G(\bu) - \G(\bv) \| \leq \LG \| \bu-\bv \|,
\end{align}
for $\bu, \bv \in \Reals^d$ and Lipschitz constant $\LG > 0$.
\end{assumption}
\noindent Note that these assumptions do not restrict the choice of $\G$, as they are met when choosing any Runge-Kutta or Taylor method \citep{hairer1993}.


In addition to assumptions on the flow maps, we require an assumption on the absolute moments of the (state-independent) random variables, which will be needed to prove \cref{thm:superlinear_bound}.
\begin{assumption}[Bounded absolute moments of $\bXi^k_n$] \label{assump:boundedRV}
For $q \geq 0$, $\tilde{r} \in \mathbb{N} \cup \{\infty \}$, and $C_2 > 0$ independent of $n$, $k$, and $\Delta T$, the $r$-th absolute moments of $\bXi^k_n(\bU^k_n)$ satisfy
\begin{align} \label{eq:moments}
    \E \big[ \| \bXi^k_n \|^r \big] \leq \big( C_2 \Delta T^{q+\frac{1}{2}}  \big)^r, \quad 1 \leq r \leq \tilde{r}.
\end{align}
\end{assumption}
\noindent This assumption enables flexibility in defining the state-independent perturbations, in the sense that it does not require that the random variables be independent, identically distributed, or centred \citep{lie2019}.
It also means that each $\bXi^k_n$ could follow a different probability distribution, with the only requirement being that they share a common maximal bound on their absolute moments with respect to the norm.
Note that we assume $\Delta T < 1$ without loss of generality here, so that for increasing $q$, the perturbations get smaller and smaller.
For $\Delta T > 1$, we can simply take $q$ to be negative. 

These assumptions will enable us to derive error bounds in the state-independent and state-dependent cases (using sampling rules 2 and 4). The sampling rule 1 and 3 case requires the following as well.
\begin{assumption}[Lipschitz exact flow] \label{assump:lipschitz_exact}
The flow map $\F$ satisfies the Lipschitz condition
\begin{align} \label{eq:lipschitz_exact}
     \| \F(\bu) - \F(\bv) \| \leq \LF \| \bu-\bv \|,
\end{align}
for $\bu, \bv \in \Reals^d$ and constant $\LF > 0$. 
\end{assumption}

\subsection{State-independent perturbations}
In this section, we derive error bounds for SParareal when using the state-independent perturbations $\bXi^k_n(\bU_n^k) = \bXi^k_n$.
\begin{theorem}[Superlinear error bound for state-independent perturbations] \label{thm:superlinear_bound}
Suppose the SParareal scheme \eqref{eq:SP1}-\eqref{eq:SP4} with $\bXi^k_n(\bU_n^k) = \bXi^k_n$ satisfies \cref{assump:exactflow,assump:coarseapprox,assump:lipschitz,assump:boundedRV}.
Then, the mean-square error \eqref{eq:MSE} of the solution to the nonlinear ODE system \eqref{eq:systemODE} at iteration $k$ and time $t_n$ satisfies
\begin{align*}
    e^k_n \leq D A^{k-1} \sum_{\ell = 0}^{n-k} \binom{\ell+k-1}{\ell} B^{\ell} + \Lambda \sum_{j=0}^{k-2} \sum_{\ell=0}^{n-(j+1)} \binom{\ell+j}{\ell} A^j B^{\ell},
\end{align*}
for $2 \leq k < n \leq N$ and constants $A = C_1^2 \Delta T^{2p+2} (2 + \Delta T^{-1})$, $B = \LG^2 (1 + 2 \Delta T)$, $\Lambda = C_2^2 \Delta T^{2q+1} (2 + \Delta T^{-1})$, and $D = A \hat{e}^0$.
\end{theorem}
\begin{proof}
Using \eqref{eq:SP4}, that $\F$ is the exact solver \eqref{eq:exact}, and adding and subtracting $\G(\bu(t_n))$, we see that
\begin{align*}
    e^{k+1}_{n+1} &= \E \big[ \| \F(\bu(t_n)) - \big( \G(\bU^{k+1}_n) + \F(\bU^k_n) - \G(\bU^k_n) + \bXi^k_n(\bU^k_n) \big) \pm\  \G(\bu(t_n)) \|^2\big] \\
    &= \E[\| W_1 + W_2 + W_3 \|^2],
\end{align*}
where $W_1$, $W_2$, and $W_3$ are given by
\begin{align*}
    W_1 &=  \F(\bu(t_n)) - \G(\bu(t_n)) - \big( \F(\bU^k_n) - \G(\bU^k_n) \big), \\
    W_2 &= \G(\bu(t_n)) - \G(\bU^{k+1}_n), \\
    W_3 &= - \bXi^k_n.
\end{align*}
Then, using the triangle inequality and \eqref{eq:youngs2} for the cross terms \citep[Sec. 4.2]{engblom2009}, we obtain
\begin{equation}\label{eq:ET1ET2ET3}
    e^{k+1}_{n+1}\leq (1 + \delta_1^{-1} + \delta_2^{-1}) \E[\|W_1\|^2] + (1 + \delta_1 + \delta_3^{-1}) \E[\|W_2\|^2] + (1 + \delta_2 + \delta_3) \E[\|W_3\|^2].
\end{equation}

Using \eqref{eq:FG2}, we can bound
\begin{align} \label{eq:ET1sq}
    \E[\|W_1\|^2] &\leq C_1^2 \Delta T^{2(p+1)} e^k_n.
\end{align}
Applying the Lipschitz condition \eqref{eq:lipschitz}, we obtain
\begin{align} \label{eq:ET2sq}
    \E[\|W_2\|^2] &\leq \LG^2 e^{k+1}_n.
\end{align}
Using \eqref{eq:moments} with $r=2$, we obtain
\begin{align} \label{eq:ET3sq}
    \E[\|W_3\|^2] &\leq C_2^2 \Delta T^{2q+1}.
\end{align}
Plugging \eqref{eq:ET1sq}-\eqref{eq:ET3sq} into \eqref{eq:ET1ET2ET3} and choosing $\delta_1 = \Delta T$, $\delta_2 = 1$, and $\delta_3 = \Delta T^{-1}$, we obtain the double recursion
\begin{align} \label{eq:rec}
    e^{k+1}_{n+1} \leq A e^k_n + B e^{k+1}_n + \Lambda, \quad e^1_{n+1} \leq D + B e^1_n,
\end{align}
where $A = C_1^2 \Delta T^{2p+2} (2 + \Delta T^{-1})$, $B = \LG^2 (1 + 2 \Delta T)$, $\Lambda = C_2^2 \Delta T^{2q+1} (2 + \Delta T^{-1})$, and $D = A \hat{e}^0$.
Solving \eqref{eq:rec} using the generating function method in \cref{lem:recursion} (see Appendix), we obtain the result.
\end{proof}

One can make alternative choices for $\delta_1$, $\delta_2$, and $\delta_3$, however, the choices given in the proof above seem to yield the smallest error bounds. 
If we were to maximise \eqref{eq:rec} over $n$, we obtain the following linear error bound in the case that $B < 1$, i.e.\ $\LG<(1+2\Delta T)^{-1/2}$.

\begin{corollary}[Linear error bound for state-independent perturbations] \label{corr:linear_bound}
Suppose the SParareal scheme \eqref{eq:SP1}-\eqref{eq:SP4} with $\bXi^k_n(\bU_n^k) = \bXi^k_n$ satisfies \cref{assump:exactflow,assump:coarseapprox,assump:lipschitz,assump:boundedRV}.
Then, the maximal mean-square error \eqref{eq:MSE2} of the solution to the nonlinear ODE system \eqref{eq:systemODE} at iteration $k$ satisfies
\begin{align*}
    \hat{e}^k \leq  \hat{e}^1 \bigg( \frac{A}{1-B} \bigg)^{k-1} + \frac{\Lambda}{1-B} \sum_{j=0}^{k-2} \bigg( \frac{A}{1-B} \bigg)^{j}, \quad \text{if} \quad B < 1,
\end{align*}
for $2 \leq k \leq N$ and constants $A = C_1^2 \Delta T^{2p+2} (2 + \Delta T^{-1})$, $B = \LG^2 (1 + 2 \Delta T)$, and $\Lambda = C_2^2 \Delta T^{2q+1} (2 + \Delta T^{-1})$.
\end{corollary}
\begin{proof}
Following the proof of \cref{thm:superlinear_bound}, we maximise \eqref{eq:rec} over $n$ to obtain 
\begin{align} \label{eq:single_rec}
    \hat{e}^{k+1} \leq \tilde{A} \hat{e}^k + \tilde{\Lambda},
\end{align}
where
\begin{align*}
    \tilde{A} = \frac{A}{1 - B} \quad \text{and} \quad \tilde{\Lambda} = \frac{\Lambda}{1 - B}.
\end{align*}
Solving recursion \eqref{eq:single_rec} with initial condition $\hat{e}^1$, we obtain the desired result.
\end{proof}

\begin{remark}
The bounds in \cref{thm:superlinear_bound} and \cref{corr:linear_bound} hold for $2 \leq  k < n \leq N$ due to the design of the SParareal scheme.
We can recover the bound for iteration $k=1$ (which is deterministic) by solving the second recursion in \eqref{eq:rec} with initial value $e^1_1=0$ such that
\begin{align} \label{eq:remark1}
    e^1_n \leq \hat{e}^0 A \sum_{i=0}^{n-2} B^i, \quad 1 \leq n \leq N.
\end{align}
For the case when $k=n$, the numerical error is zero, as $\F$ will have propagated the exact initial condition at $t_0$ forward $k$ times without any perturbations, just like parareal (recall \cref{rem:parareal}).
\end{remark}

\begin{remark} \label{rem:loose_bound}
The bound in \cref{thm:superlinear_bound} (similarly for \cref{corr:linear_bound}) can be written as
\begin{align} \label{eq:loose_superlinear_bound}
    e^k_n \leq C_{k,n} \max \{ \Delta T^{(2p+1)k}, \Delta T^{2q} \},
\end{align}
where $C_{k,n}$ is a function of $n$ and $k$.
Assuming $\Delta T < 1$ and that $C_{k,n}$ is non-increasing in $k$, the accuracy of SParareal should increase with each iteration proportional to the local truncation error of $\mathcal{G}$ (i.e.\ the term $\Delta T^{(2p+1)k}$) up until the errors induced by the perturbations (i.e. $\Delta T^{2q})$) become dominant.
We illustrate this property numerically in \cref{sec:numerics}.
\end{remark}

\begin{remark}
As $\Delta T \rightarrow 0$, both error bounds go to zero as expected, as can be seen clearly in \eqref{eq:loose_superlinear_bound}.
The intuition being that as $\Delta T \rightarrow 0$, the local truncation error of $\G$ goes to zero, i.e.\ it `becomes' the exact flow map $\F$, see \eqref{eq:FG2}.
\end{remark}

\begin{remark}
If we send $q \rightarrow \infty$ (assuming $\Delta T < 1$), the second moments of the random variables vanish and we recover the classic parareal scheme. 
This can be seen in both \cref{thm:superlinear_bound} and \cref{corr:linear_bound}, where $\Lambda$ vanishes as $q \rightarrow \infty$, leading to bounds similar to those for classic parareal.
These bounds are not identical to those calculated by \cite{gander2007}, \cite{gander2008}, and \cite{gander2022} because we are working with the mean-square error, not the (mean) absolute error.
\end{remark}

\begin{remark}
If we additionally assume that the random variables $\bXi^k_n$ are centred, i.e.\ $\E[ \bXi^k_n ] = \bm{0}$, and work in the $2$-norm, i.e.\ $\| \bu \|^2_2 = \langle \bu, \bu \rangle = \sum_{i=1}^d u_i^2$, in the proof of \cref{thm:superlinear_bound}, we can write \eqref{eq:ET1ET2ET3} as
\begin{align*}
    e^{k+1}_{n+1}\leq (1 + \delta_1^{-1}) \E[\|W_1\|^2] + (1 + \delta_1) \E[\|W_2\|^2] + \E[\|W_3\|^2] + 2 \E[\langle W_1,W_3 \rangle] + 2 \E[\langle W_2,W_3 \rangle],
\end{align*}
where the final two terms are equal to zero by independence of $W_1$ and $W_2$ with $W_3$ and using the fact that each $\bXi^k_n$ is centred.
Continuing the proof, we obtain the same bounds for \cref{thm:superlinear_bound} and \cref{corr:linear_bound} with slightly altered constants  $A = C_1^2 \Delta T^{2p+2} (1 + \Delta T^{-1})$, $B = \LG^2 (1 + \Delta T)$, $\Lambda = C_2^2 \Delta T^{2q+1}$, and $D = A \hat{e}^0$. 
\end{remark}


\subsection{State-dependent perturbations (sampling rules)}
We now use the previous results to derive the corresponding error bounds for the state-dependent sampling rules defined in \cref{table:1}.

\begin{corollary}[Linear error bound for state-dependent sampling rules 2 and 4] \label{corr:SR_24}
Suppose the SParareal scheme \eqref{eq:SP1}-\eqref{eq:SP4} satisfies \cref{assump:exactflow,assump:coarseapprox,assump:lipschitz}, with $\bXi^k_n(\bU_n^k)$ defined using sampling rule 2 or 4. 
Then, the maximal mean-square error \eqref{eq:MSE2} of the solution to the nonlinear ODE system \eqref{eq:systemODE} at iteration $k$ satisfies
\begin{align*}
    \hat{e}^k \leq \hat{e}^0 \biggr[ \frac{ A + \Lambda_1 + \sqrt{(A + \Lambda_1)^2 + 4 \Lambda_2 (1 - B)}}{2 (1 - B)} \biggr]^k, \quad \text{if} \quad B < 1,
\end{align*}
for $ 2 \leq k \leq N$ and constants $A = C_1^2 \Delta T^{2p+2} (2 + \Delta T^{-1})$, $B = \LG^2 (1 + 2 \Delta T)$, $\Lambda_1 = C_1^2 \Delta T^{2p+2} \LG^2 (1 + \Delta T^{-1})$, and $\Lambda_2 = C_1^2 \Delta T^{2p+2} \LG^2 (1 + \Delta T)$.
\end{corollary}
\begin{proof}[Proof for sampling rule 2]
The proof follows in the same fashion as \cref{thm:superlinear_bound}. Instead of using the bound \eqref{eq:ET3sq}, we obtain, using \eqref{eq:Xi} and applying \eqref{eq:FG2},
\begin{align} \label{eq:E_alp}
    \E [ \| W_3 \|^2 ] &\leq C_1^2 \Delta T^{2(p+1)} \E [\| \bm{\alpha}^k_n - \bU^k_n \|^2 ].
\end{align}
Substituting in $\bm{\alpha}^k_n$ for sampling rule 2 (\cref{table:1}) we get
\begin{align}
    \E [ \| W_3 \|^2 ] &\leq C_1^2 \Delta T^{2(p+1)} \E [\| \bm{\sigma}^k_n \circ \bm{z}^k_n \|^2 ] \nonumber \\
    &\leq C_1^2 \Delta T^{2(p+1)} \E [\| \bm{\sigma}^k_n \|^2 ] \E [ \| \bm{z}^k_n \|^2 ] \nonumber \\
    &\leq C_1^2 \Delta T^{2(p+1)} \LG^2 \E [\| \bU^{k}_{n-1} - \bU^{k-1}_{n-1} \|^2]. \nonumber
\end{align}
The second inequality follows by Cauchy-Schwarz and independence of $\bm{\sigma}^k_n$ and $\bm{z}^k_n$.
The third follows by plugging in $\bm{\sigma}^k_n$, applying \eqref{eq:lipschitz}, and noting that $\E[\| \bm{z}^k_n \|^2] = 1$.
Next, we add and subtract $\bu(t_{n-1})$ inside the expectation and then apply \eqref{eq:youngs2}, with $\delta = \Delta T$, to get
\begin{align} \label{eq:sig_z}
    \E [ \| W_3 \|^2 ] \leq C_1^2 \Delta T^{2(p+1)} \LG^2 \big( (1 + \Delta T^{-1}) e^k_{n-1} +  (1 +  \Delta T) e^{k-1}_{n-1} \big). 
\end{align}
Using the new bound for $\E[\| W_3 \|^2]$ in \eqref{eq:ET1ET2ET3}, we obtain the double recurrence
\begin{align} \label{eq:5-term1}
    e^{k+1}_{n+1} \leq A e^k_n + B e^{k+1}_n + \Lambda_1 e^{k}_{n-1} + \Lambda_2 e^{k-1}_{n-1},
\end{align}
where $A = C_1^2 \Delta T^{2p+2} (2 + \Delta T^{-1})$, $B = \LG^2 (1 + 2 \Delta T)$, $\Lambda_1 = C_1^2 \Delta T^{2p+2} \LG^2 (1 + \Delta T^{-1})$, and $\Lambda_2 = C_1^2 \Delta T^{2p+2} \LG^2 (1 + \Delta T)$.
Maximising over $n$, we obtain
\begin{align} \label{eq:three_term}
    \hat{e}^{k+1} \leq \tilde{A} \hat{e}^k + \tilde{B} \hat{e}^{k-1},
\end{align}
where
\begin{align*}
    \tilde{A} = \frac{A + \Lambda_1}{1 - B} \quad \text{and} \quad \tilde{B} = \frac{\Lambda_2}{1 - B}.
\end{align*}
Recursion \eqref{eq:three_term} can be solved using \cref{lem:recursion2} (see Appendix), resulting in the desired bound.
\end{proof}

\begin{proof}[Proof for sampling rule 4]
The proof follows in the same way as the proof for sampling rule 2, except that $\E[\| \sqrt{3} (2 \bm{w}^k_n - \bm{1}) \|^2] = 1$ is used in place of $\E[\| \bm{z}^k_n \|^2] = 1$.
\end{proof}

\begin{corollary}[Linear error bound for state-dependent sampling rules 1 and 3] \label{corr:SR_13}
Suppose the SParareal scheme \eqref{eq:SP1}-\eqref{eq:SP4} satisfies \cref{assump:exactflow,assump:coarseapprox,assump:lipschitz,assump:lipschitz_exact}, with $\bXi^k_n(\bU_n^k)$ defined using sampling rule 1 or 3. 
Then, the maximal mean-square error \eqref{eq:MSE2} of the solution to the nonlinear ODE system \eqref{eq:systemODE} at iteration $k$ satisfies
\begin{align*}
    \hat{e}^k \leq \hat{e}^0 \biggr[ \frac{ A + \Lambda_1 + \Lambda_3 + \sqrt{(A + \Lambda_1 + \Lambda_3)^2 + 4 \Lambda_2 (1 - B)}}{2 (1 - B)} \biggr]^k, \quad \text{if} \quad B < 1,
\end{align*}
for $2 \leq k \leq N$ and constants $A = C_1^2 \Delta T^{2p+2} (2 + \Delta T^{-1})$, $B = \LG^2 (1 + 2 \Delta T)$, $\Lambda_1 = 2 C_1^2 \Delta T^{2p+2} \LG^2 (1 + \Delta T^{-1})$, $\Lambda_2 = 2 C_1^2 \Delta T^{2p+2} ( \LG^2 (1 + \Delta T) + 2 \LF^2 )$, and $\Lambda_3 = 4 C_1^2 \Delta T^{2p+2}$.
\end{corollary}

\begin{proof}[Proof for sampling rule 1]
The proof follows in the same fashion as \cref{corr:SR_24}. 
Substituting $\bm{\alpha}^k_n$ for sampling rule 1 (\cref{table:1}) in \eqref{eq:E_alp}, we get
\begin{align}
    \E [ \| W_3 \|^2 ] &\leq C_1^2 \Delta T^{2(p+1)} \E [\| \F(\bU^{k-1}_{n-1}) - \bU^k_{n} + \bm{\sigma}^k_n \circ \bm{z}^k_n \|^2 ] \nonumber \\
    &\leq 2 C_1^2 \Delta T^{2(p+1)} \big( \underbrace{\E [\| \F(\bU^{k-1}_{n-1}) - \bU^k_{n} \|^2]}_{\text{1st Term}} + \underbrace{\E [\| \bm{\sigma}^k_n \circ \bm{z}^k_n \|^2 ]}_{\text{2nd Term}} \big) \label{eq:SR1_full}.
\end{align}
The second inequality follows by applying \eqref{eq:youngs2} with $\delta = 1$.
To bound the first term in \eqref{eq:SR1_full}, we add and subtract $\F(\bu(t_{n-1}))$ inside the expectation and apply \eqref{eq:youngs2} again with $\delta = 1$, obtaining
\begin{align*}
    \text{1st Term} &\leq 2 \big( \E [\| \F(\bU^{k-1}_{n-1}) - \F(\bu(t_{n-1})) \|^2] + \E [\| \F(\bu(t_{n-1})) - \bU^k_{n} \|^2 ] \big) \\
    &\leq 2 \big( \LF^2 e^{k-1}_{n-1} + e^k_n \big).
\end{align*}
The second inequality follows by applying the Lipschitz condition \eqref{eq:lipschitz_exact} and recalling that $\F$ is the exact solver \eqref{eq:exact}.
The second term in \eqref{eq:SR1_full} can be bounded as in \eqref{eq:sig_z} in \cref{corr:SR_24},
\begin{align*}
    \text{2nd Term} \leq \LG^2 \big( (1 + \Delta T^{-1}) e^k_{n-1} +  (1 +  \Delta T) e^{k-1}_{n-1} \big).
\end{align*}
Combining both terms in \eqref{eq:SR1_full}, we obtain 
\begin{align*}
    \E [ \| W_3 \|^2 ] &\leq \Lambda_1 e^k_{n-1} + \Lambda_2 e^{k-1}_{n-1} + \Lambda_3 e^k_n,
\end{align*}
where $\Lambda_1 = 2 C_1^2 \Delta T^{2p+2} \LG^2 (1 + \Delta T^{-1})$, $\Lambda_2 = 2 C_1^2 \Delta T^{2p+2} ( \LG^2 (1 + \Delta T) + 2 \LF^2 )$, and $\Lambda_3 = 4 C_1^2 \Delta T^{2p+2}$.
Using the new bound for $\E[\| W_3 \|^2]$ in \eqref{eq:ET1ET2ET3}, we obtain the following recurrence
\begin{align} \label{eq:5-term2}
    e^{k+1}_{n+1} \leq (A + \Lambda_3) e^k_n + B e^{k+1}_n + \Lambda_1 e^{k}_{n-1} + \Lambda_2 e^{k-1}_{n-1},
\end{align}
where $A = C_1^2 \Delta T^{2p+2} (2 + \Delta T^{-1})$ and $B = \LG^2 (1 + 2 \Delta T)$.
Maximising over $n$, we obtain
\begin{align} \label{eq:three_term2}
    \hat{e}^{k+1} \leq \tilde{A} \hat{e}^k + \tilde{B} \hat{e}^{k-1},
\end{align}
where
\begin{align*}
    \tilde{A} = \frac{A + \Lambda_1 + \Lambda_3}{1 - B} \quad \text{and} \quad \tilde{B} = \frac{\Lambda_2}{1 - B}.
\end{align*}
Recursion \eqref{eq:three_term2} can be solved using \cref{lem:recursion2} (see Appendix), resulting in the desired bound.
\end{proof}

\begin{proof}[Proof for sampling rule 3]
The proof follows in a similar fashion to the proof for sampling rule 1, with $\E[\| \sqrt{3} (2 \bm{w}^k_n - \bm{1}) \|^2] = 1$ being used in place of $\E[\| \bm{z}^k_n \|^2] = 1$.
\end{proof}

\begin{remark} \label{rem:num_bound}
In \cref{sec:numerics}, we can observe the behaviour of $e^k_n$ (not just $\hat{e}^k$) for each of the sampling rules by solving the recursions \eqref{eq:5-term1} and \eqref{eq:5-term2} numerically.
We do this by replacing the inequality with an equality, i.e.\ upper bounding the error estimate.
\end{remark}

\section{Numerical Experiments} \label{sec:numerics}
Here, we present some experiments to compare the theoretical bounds derived in \cref{sec:convergence} with the errors generated by running SParareal numerically.
MATLAB code to reproduce the results below can be found at \url{https://github.com/kpentland/StochasticPararealAnalysis}.

\subsection{System of linear ODEs} \label{sec:linear_ODEs}
\begin{figure}[b!]
    \centering
    \begin{subfigure}{0.36\linewidth}
        \includegraphics[width=\textwidth]{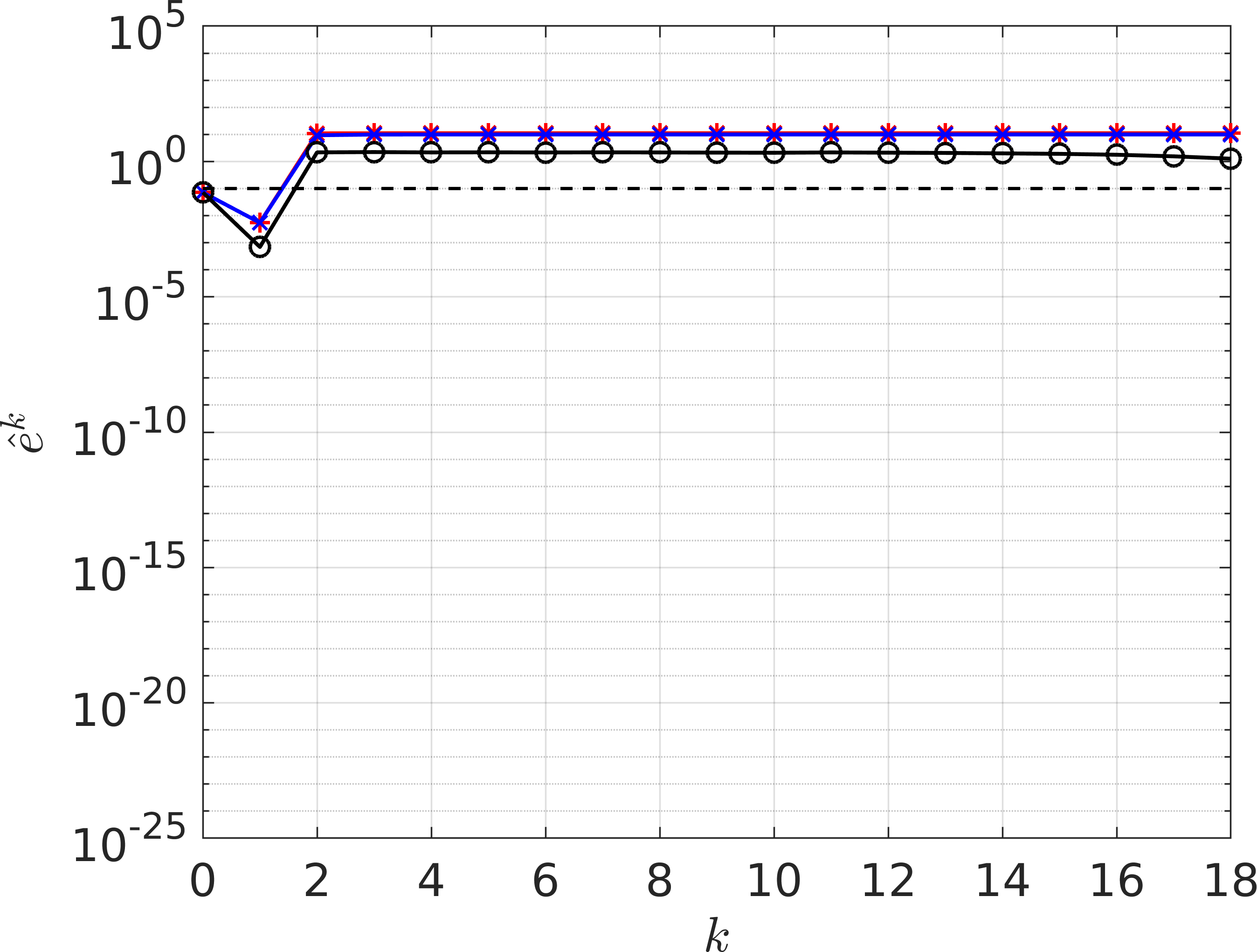}
        \caption{$q=0$}
    \end{subfigure}
    \begin{subfigure}{0.31\linewidth}
        \includegraphics[width=\textwidth]{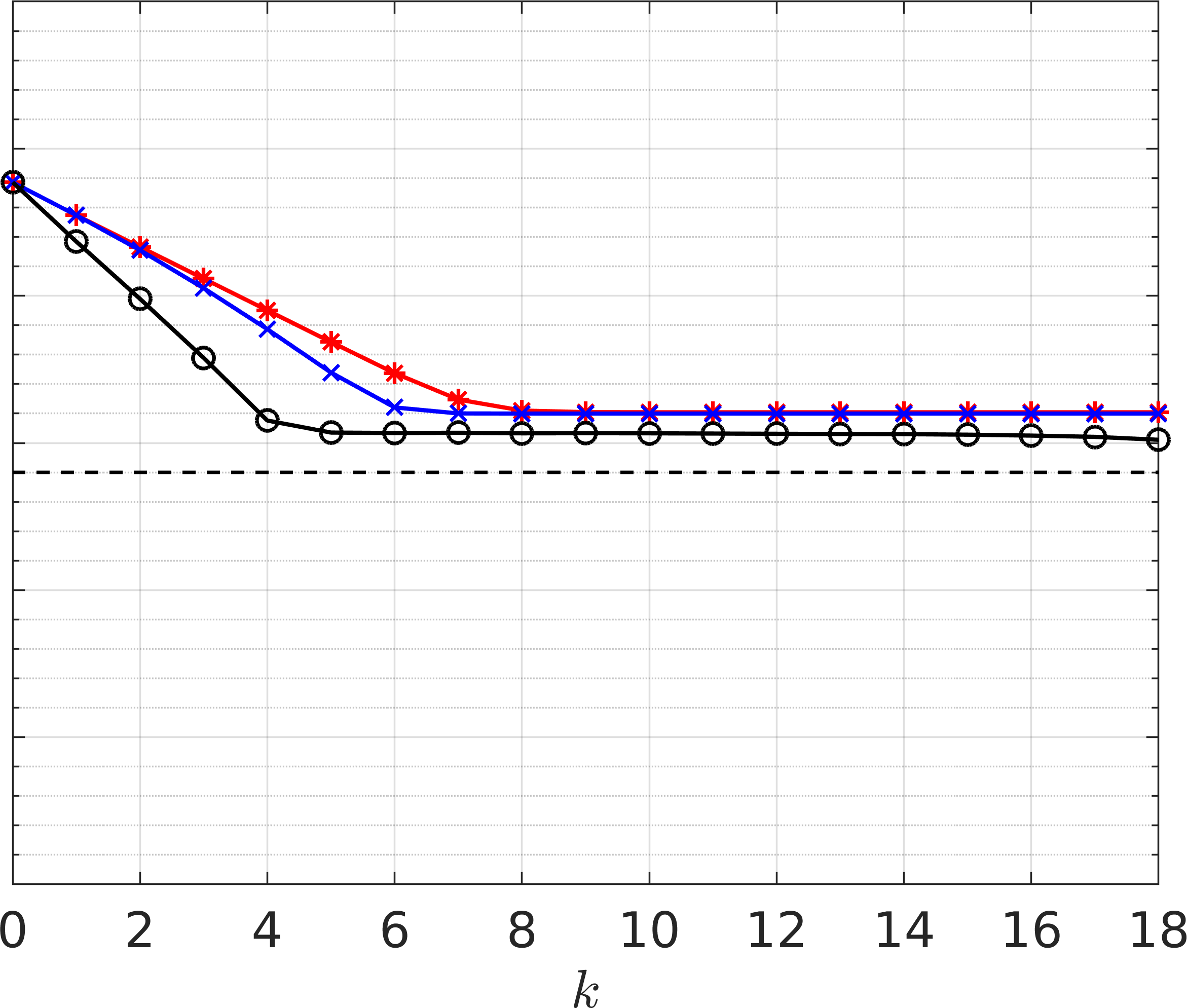}
        \caption{$q=5$}
    \end{subfigure}
    \begin{subfigure}{0.31\linewidth}
        \includegraphics[width=\textwidth]{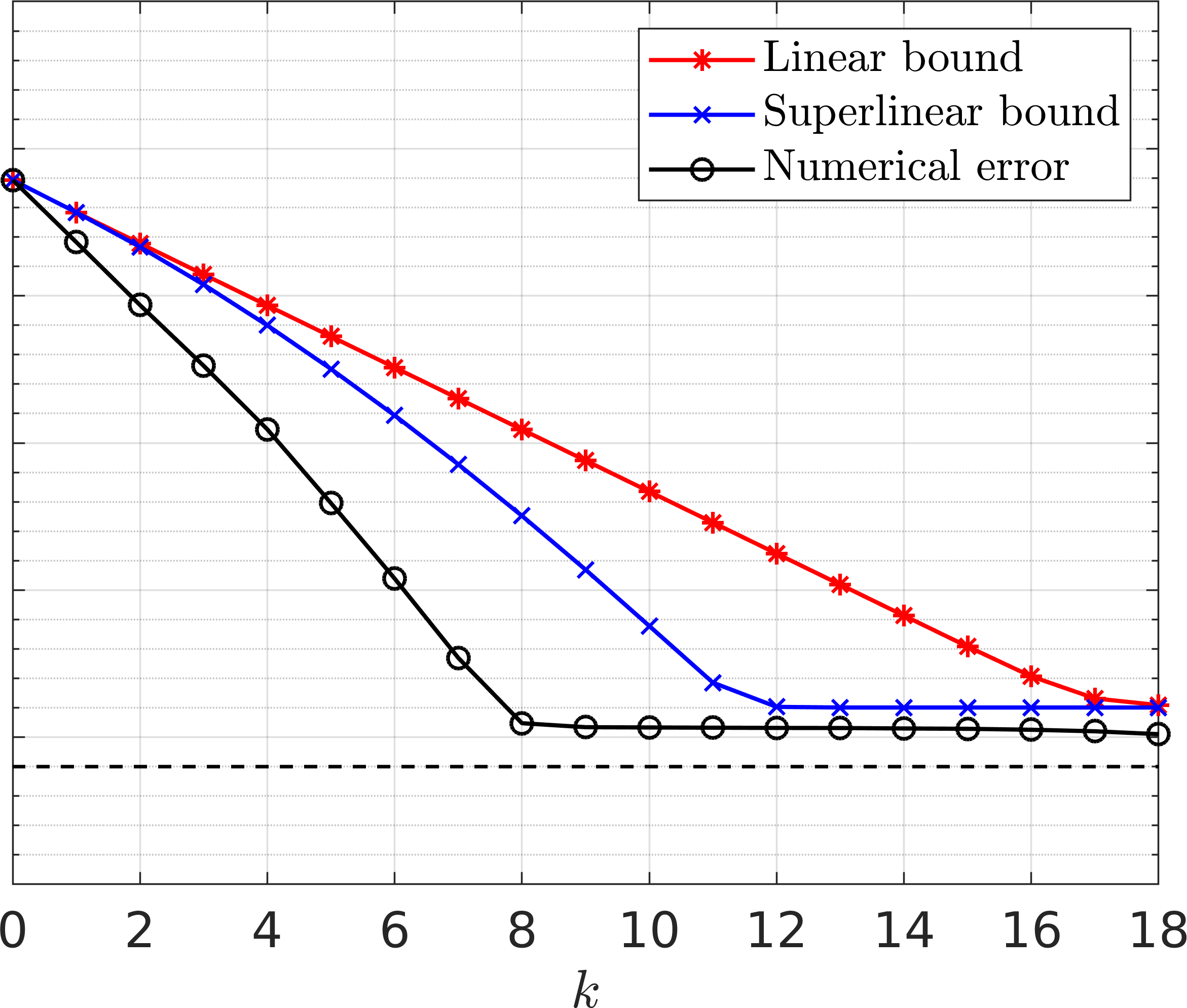}
        \caption{$q=10$}
    \end{subfigure}
    \caption{Theoretical bounds vs. numerical errors for SParareal applied to the linear system of ODEs \eqref{eq:linear_system_ODE} (with $B<1$) using state-independent Gaussian perturbations \eqref{eq:Gaussians}.
    The superlinear bound (\cref{thm:superlinear_bound}) is given in blue, the linear bound (\cref{corr:linear_bound}) in red, the numerical error in black, and $\Delta T^{2q+1}$ in dashed black. 
    Each plot corresponds to a different level of Gaussian noise: (a) $q=0$, (b) $q=5$, and (c) $q=10$.
    Numerical errors were calculated by averaging over $500$ independent realisations of SParareal.}
    \label{fig:linear_conv_case1}
\end{figure}
\begin{figure}[t!]
    \centering
    \begin{subfigure}{0.35\linewidth}
        \includegraphics[width=\textwidth]{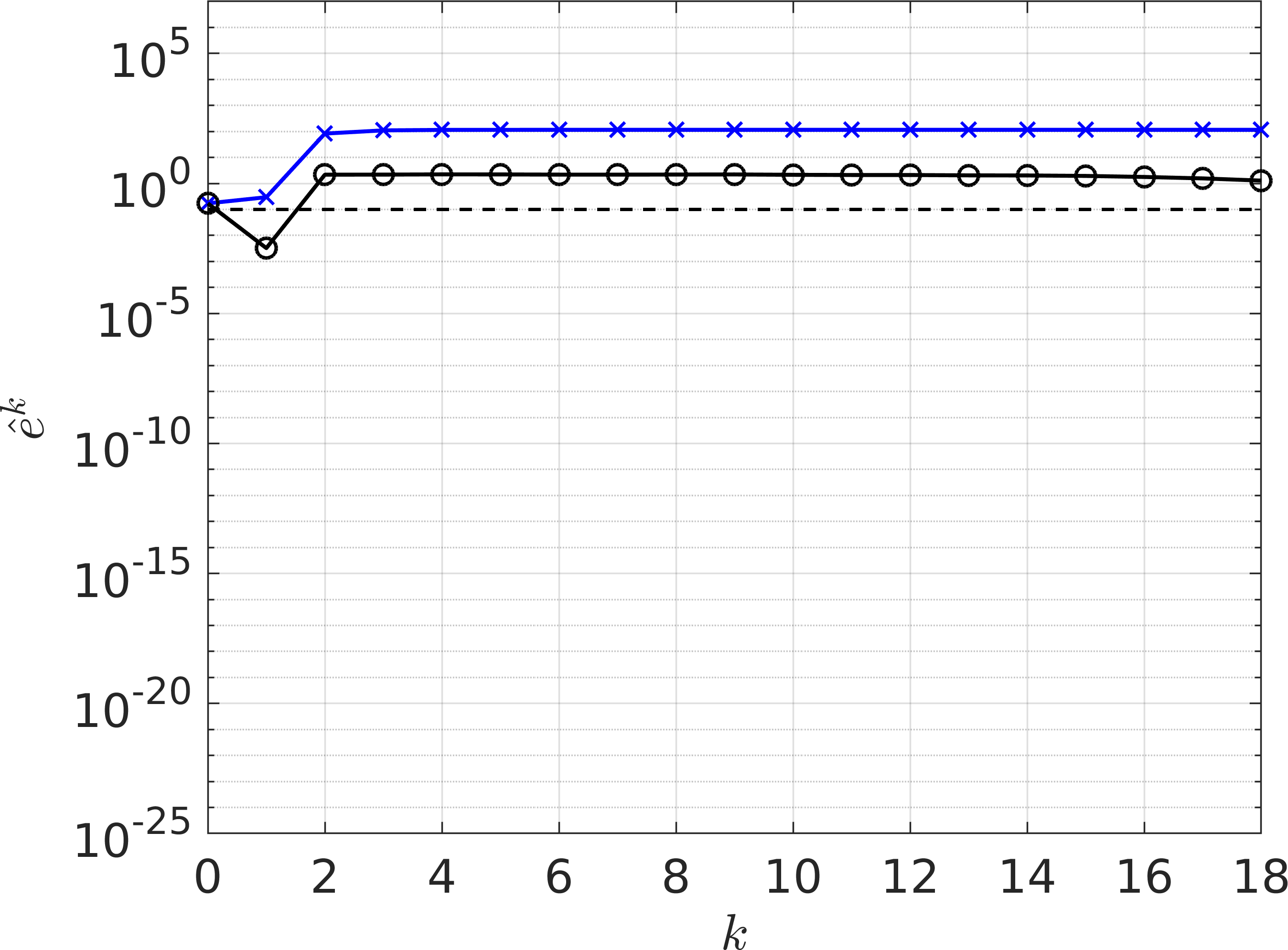}
        \caption{$q=0$}
    \end{subfigure}
    \begin{subfigure}{0.31\linewidth}
        \includegraphics[width=\textwidth]{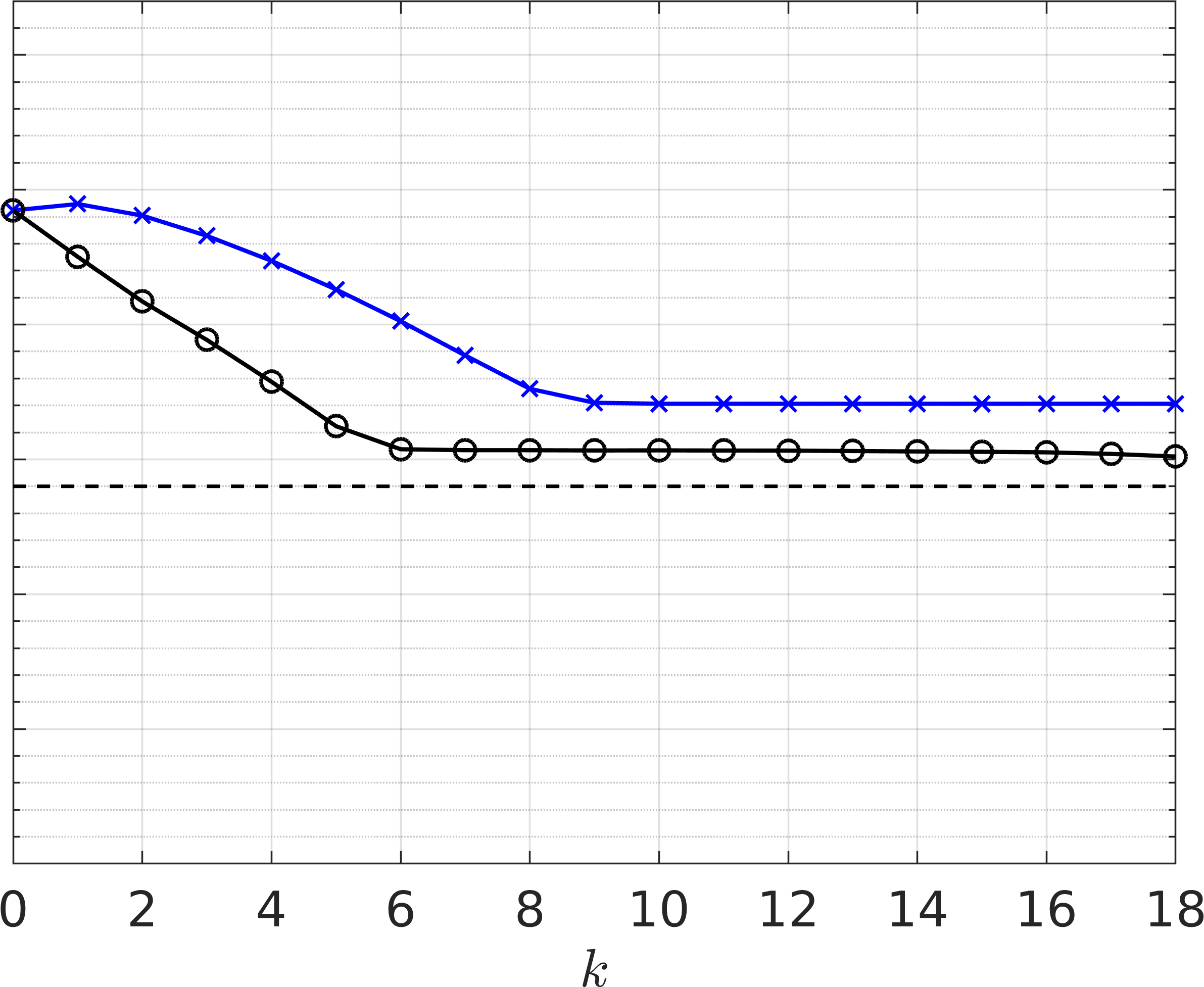}
        \caption{$q=5$}
    \end{subfigure}
    \begin{subfigure}{0.31\linewidth}
        \includegraphics[width=\textwidth]{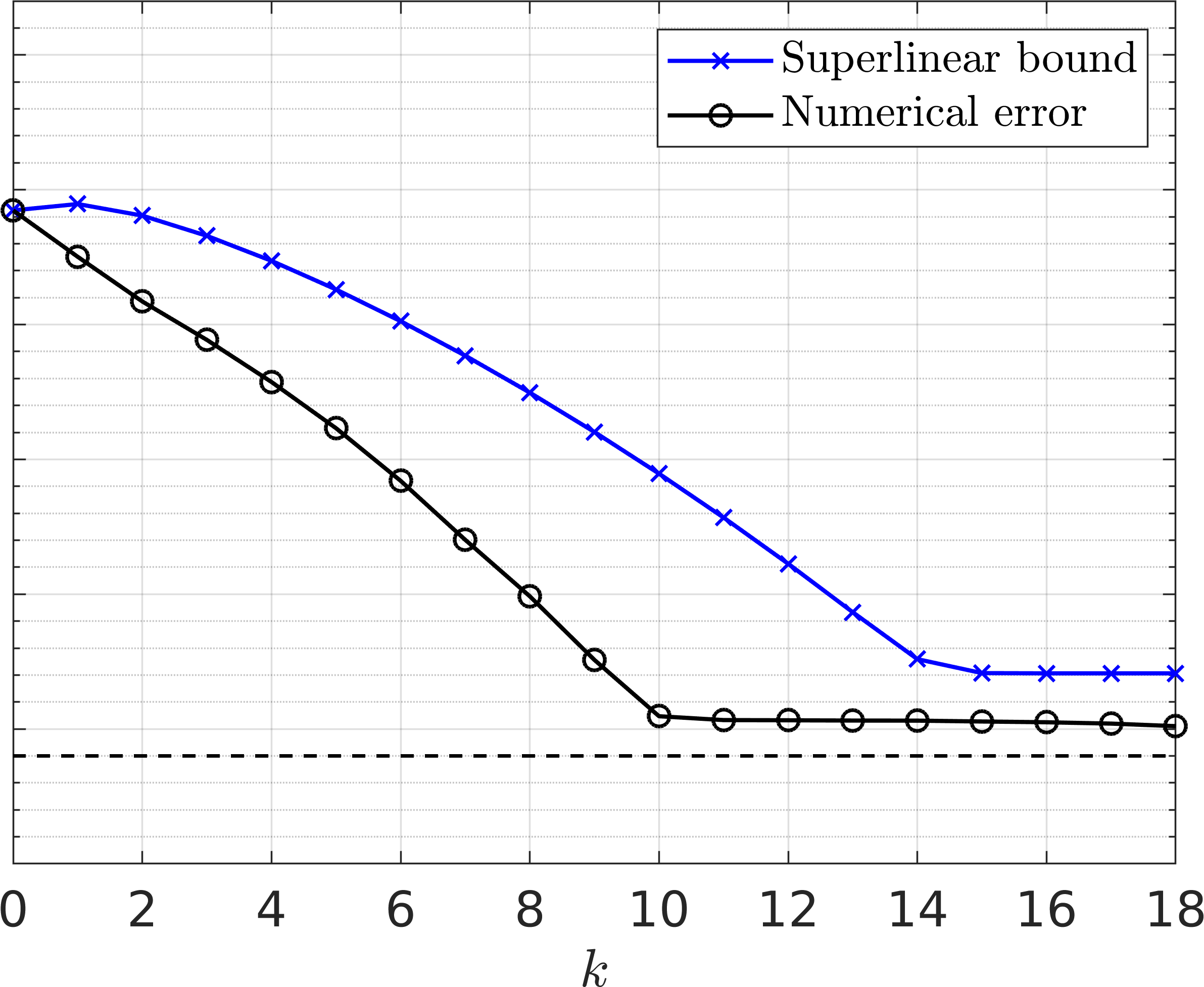}
        \caption{$q=10$}
    \end{subfigure}
    \caption{Theoretical bounds vs. numerical errors for SParareal applied to the linear system of ODEs \eqref{eq:linear_system_ODE} (with $B\geq1$) using state-independent Gaussian perturbations \eqref{eq:Gaussians}.
    The superlinear bound (\cref{thm:superlinear_bound}) is given in blue, the numerical error in black, and $\Delta T^{2q+1}$ in dashed black. 
    Each plot corresponds to a different level of Gaussian noise: (a) $q=0$, (b) $q=5$, and (c) $q=10$.
    Numerical errors were calculated by averaging over $500$ independent realisations of SParareal.}
    \label{fig:linear_conv_case2}
\end{figure}
\begin{figure}[b!]
    \centering
    \includegraphics[width=0.49\linewidth]{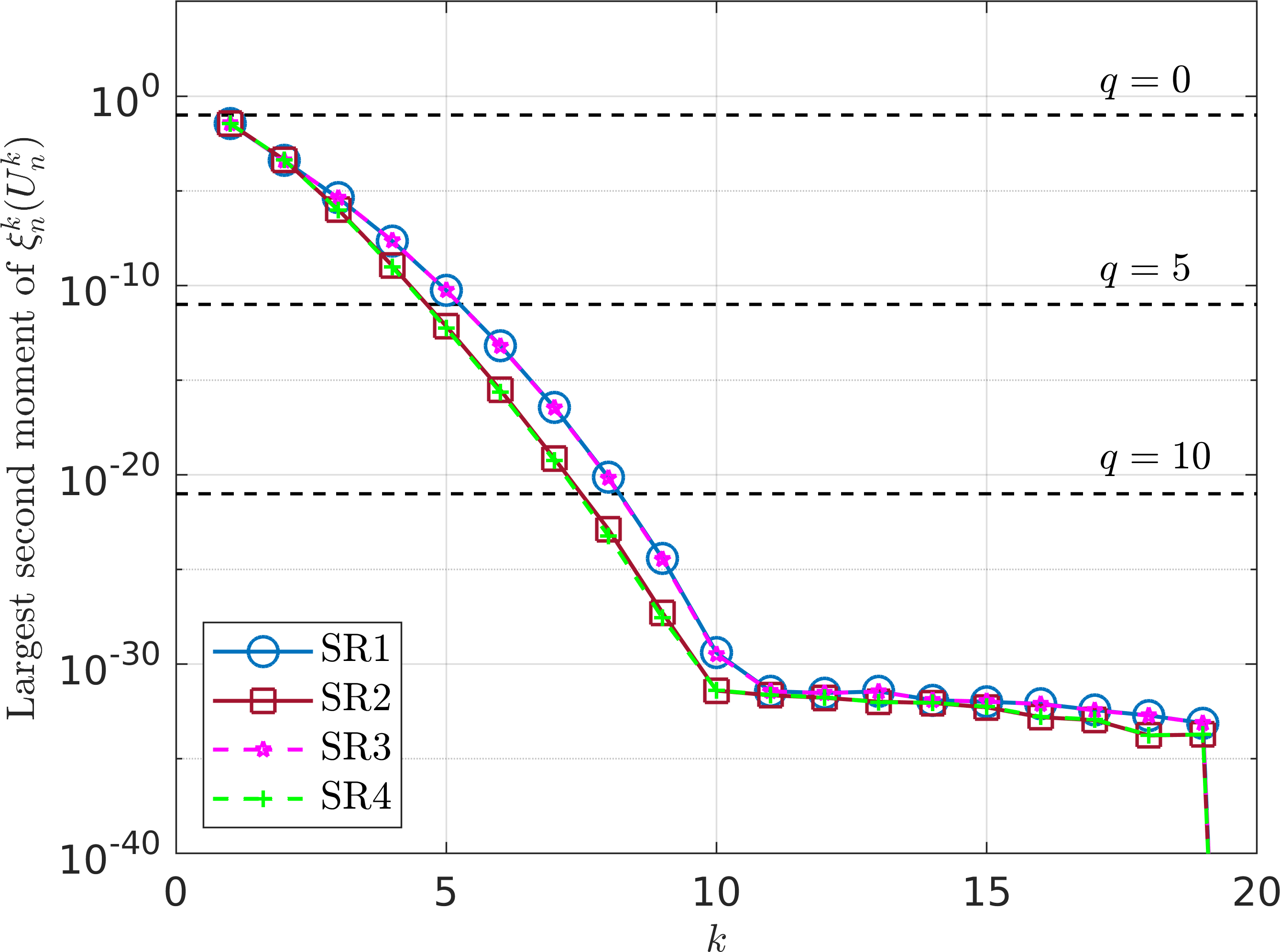}
    \caption{Largest second moments (over $n$) of $\bXi^k_n(\bU^k_n)$ fo the sampling rules $1$ to $4$ (light blue, brown, purple, and green respectively) and the Gaussian perturbations \eqref{eq:Gaussians} for $q \in \{0,5,10 \}$ (dashed black), plotted against iteration number $k$.
    Second moments for the sampling rules were calculated by averaging over $500$ independent realisations of SParareal.}
    \label{fig:second_moments}
\end{figure}
\begin{figure}[t!]
    \centering
    \begin{subfigure}{0.48\linewidth}
        \includegraphics[width=\textwidth]{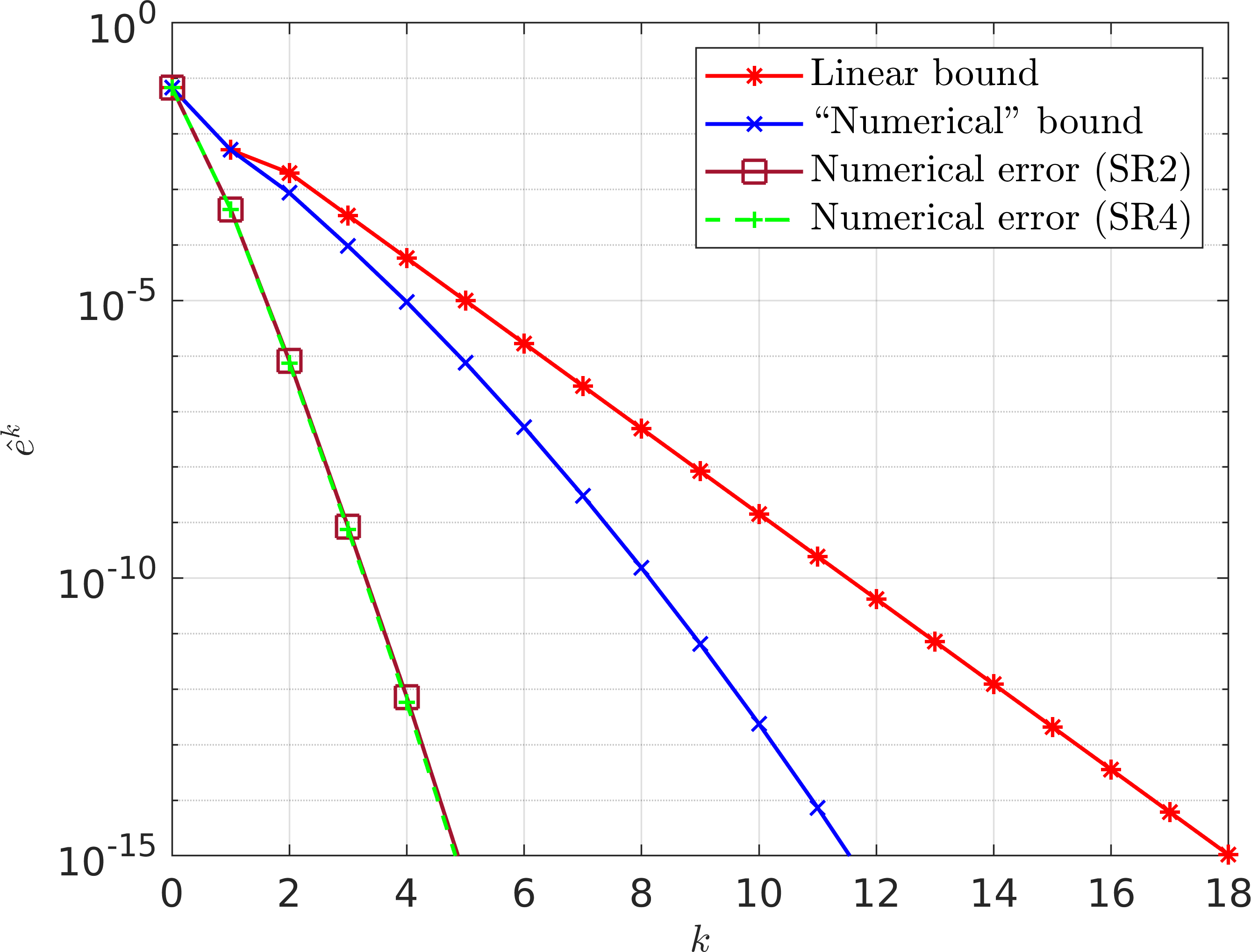}
        \caption{Sampling rules 2 and 4}
    \end{subfigure}
    \begin{subfigure}{0.49\linewidth}
        \includegraphics[width=\textwidth]{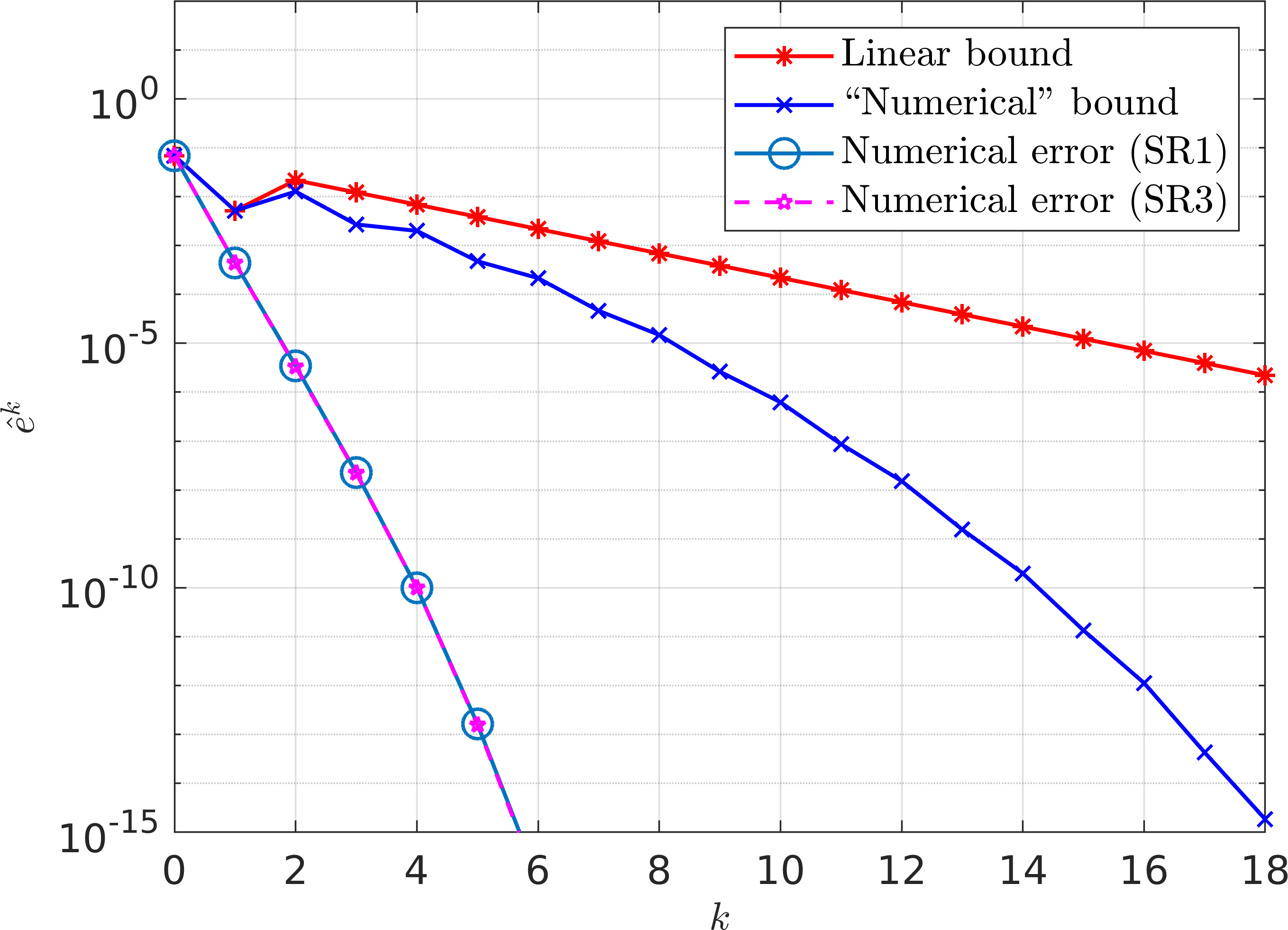}
        \caption{Sampling rules 1 and 3}
    \end{subfigure}
    \caption{Theoretical bounds vs. numerical errors for SParareal applied to the linear system of ODEs \eqref{eq:linear_system_ODE} (with $B<1$) using the state-dependent sampling rules.
    (a) The linear bound in \cref{corr:SR_24} is shown in red, the numerically solved recursion \eqref{eq:5-term1} in blue, and the numerical errors for sampling rules 2 and 4 in brown and green, respectively. 
    (b) The linear bound in \cref{corr:SR_13} is shown in red, the numerically solved recursion \eqref{eq:5-term2} in blue, and the numerical errors for sampling rules 1 and 3 in light blue and purple, respectively.
    Numerical errors were calculated by averaging over $500$ independent realisations of SParareal.}
    \label{fig:linear_conv_SR}
\end{figure}
\begin{figure}[b!]
    \centering
    \includegraphics[width=0.49\linewidth]{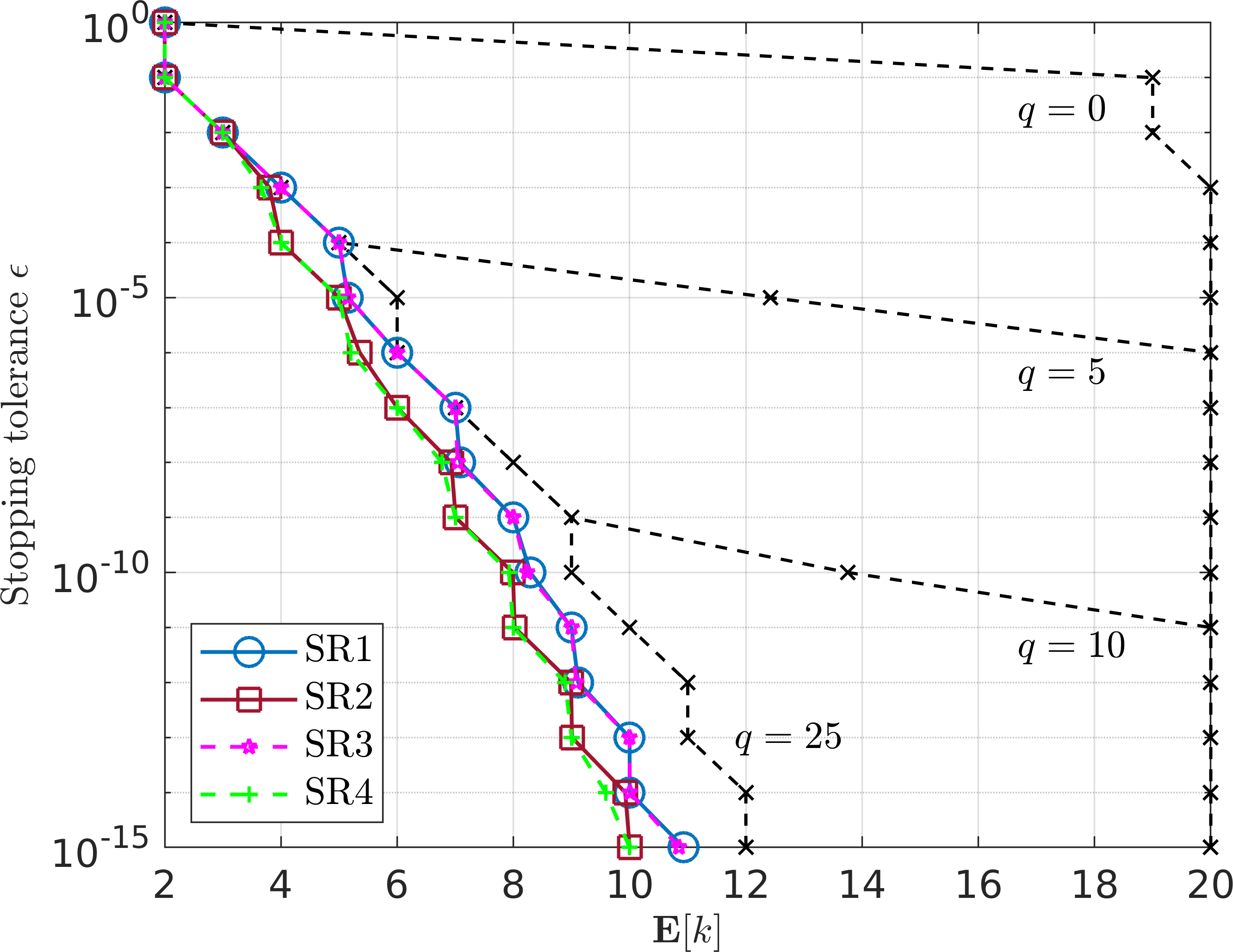}
    \caption{Expected number of iterations $k$ taken to reach stopping tolerance $\varepsilon$ \eqref{eq:tolerance} for SParareal applied to the linear system \eqref{eq:linear_system_ODE} (with $B<1$).
    Results plotted using SParareal with each sampling rule (see legend) and the Gaussian perturbations \eqref{eq:Gaussians} for $q \in \{0,5,10,25 \}$ (dashed black lines).
    $\mathbb{E}[k]$ calculated by averaging $k$ over $500$ independent realisations of SParareal.}
    \label{fig:epsilon_vs_k}
\end{figure}
In the following experiments, we solve the linear system of ODEs
\begin{equation} \label{eq:linear_system_ODE} 
    \frac{\rd \bm{u}}{\rd t} = Q \bm{u} \quad \text{over} \quad t \in [0,T], \quad \text{with} \quad \bm{u}(0) = \bm{u}_0,
\end{equation}
where $Q \in \Reals^{d \times d}$.
This system has the exact solution $\bm{u}(t) = \bm{u}_0 e^{Qt}$, where $e^{Qt} = \sum_{i=0}^{\infty} (Qt)^i /i!$ is the matrix exponential.

First, we examine the superlinear and linear bounds derived in \cref{thm:superlinear_bound} and \cref{corr:linear_bound}, respectively, by running SParareal numerically with state-independent Gaussian perturbations
\begin{align} \label{eq:Gaussians}
    \bXi^k_n \sim \mathcal{N}(\bm{0},\Delta T^{2q+1} I_d), \quad q \geq 0.
\end{align}
We solve \eqref{eq:linear_system_ODE} with $d = 100$ and $T=2$, discretising the time interval into $N = 20$ time slices so that $\Delta T = 0.25$.
We construct the matrix of coefficients $Q$ such that $B<1$ and also select initial conditions $\bm{u}_0 \in [-5,5]^d$.
We use the exact solver $\F(\bu) = \bu e^{Q \Delta T}$ and the forward Euler method $\G(\bu) = (I_d + Q \Delta T) \bu$.
In \cref{fig:linear_conv_case1}, we plot the maximal theoretical bounds $\hat e^k$ and numerical errors of SParareal as a function of $k$ for different values of $q$ when $B<1$. 
These results illustrate how the errors decrease as $k$ increases (except when $q=0$), up until the error induced by the perturbations become dominant---exactly the effect described in \cref{rem:loose_bound}.
For all considered values of $q$, the error for $k\geq 2$ has a hard lower bound of $\mathcal{O}(\Delta T^{2q+1})$, i.e.\ the error cannot go below the second moments of the perturbations (indicated by the dashed black line in each case).
By altering $Q$ and running the same experiment, we see similar effects in the $B \geq 1$ case, see \cref{fig:linear_conv_case2}.

It can be seen that, regardless of $B$, using the state-independent perturbations may not be optimal because of the lower bound forced upon the errors.
If they are to be used, then they would need to be chosen such that the second moments are smaller than the accuracy of the solutions sought.
This approach, however, may not yield accelerated convergence over the classic parareal scheme.
To avoid this (and the lower bound on accuracy), the perturbations need to be state-dependent and therefore able to adapt, i.e. the second moments need to decrease with $k$ and scale with $n$.
In \cref{fig:second_moments}, we illustrate how the second moments of the perturbations used in the state-dependent sampling rules decrease with $k$ throughout the SParareal simulation, comparing these to the fixed second moments of the Gaussians \eqref{eq:Gaussians} for each $q \in \{0,5,10 \}$ (dashed lines).
Using the sampling rules enables SParareal to sample from probability distributions that begin to ``contract'' around the exact solution states as the simulation progresses, i.e.\ as $k$ increases. 
This results in high solution accuracy in very few iterations, as will be shown in \cref{fig:linear_conv_SR}.

It should be noted that we could have also chosen a different distribution other than the Gaussian from which to sample each state-independent $\bXi^k_n$, as long as \cref{assump:boundedRV} is satisfied.
For example, choosing uniformly distributed perturbations $\bXi^k_n \sim \mathcal{U}[-\sqrt{3}\Delta T^{q+\frac{1}{2}},\sqrt{3}\Delta T^{q+\frac{1}{2}}]^d $ yielded almost identical results (not shown).

Next, we plot the linear bounds for perturbations defined by the sampling rules, i.e. \cref{corr:SR_24} and \cref{corr:SR_13}, against the corresponding numerical errors in \cref{fig:linear_conv_SR} (for the $B < 1$ problem).
We observe that the linear bounds are not that tight due to the maximisation over $n$ required to derive them.
However, by solving recursions \eqref{eq:5-term1} and \eqref{eq:5-term2} numerically (recall \cref{rem:num_bound}), we observe a tighter bound on the error. 
All that is required to calculate these ``numerical'' bounds are the errors at the `zeroth' iteration (obtained from SParareal itself by just running $\mathcal{G}$), errors at the first iteration (recall \eqref{eq:remark1}) and the constants $C_1$ and $\LG$.
Note that the numerical errors for sampling rules 2/4 and 1/3 overlap because the perturbations used in each scheme have almost identical second moments (recall \cref{fig:second_moments}). 

In \cref{fig:epsilon_vs_k}, we compare the performance of the state-independent and -dependent perturbations by plotting the expected number of iterations $\mathbb{E}[k]$ taken to reach a pre-defined stopping tolerance $\varepsilon$, recall \eqref{eq:tolerance}.
We observe that, on average, the sampling rules reach tolerance in fewer iterations than the state-independent perturbations.
The sampling rules also outperform classic parareal, which can be observed by comparing them to the state-independent perturbations for $q=25$, for which perturbations are so small that SParareal is practically deterministic (i.e.\ parareal).
Recall that reducing $k$ by even a few iterations can significantly increase parallel speedup.
The other advantage of using SParareal is that it returns a distribution of solution trajectories upon multiple realisations, instead of a single solution trajectory as parareal does, which can be interpreted as some form of uncertainty quantification over the solution to the IVP.

\subsection{Scalar nonlinear ODE} \label{sec:nonlinear_ODEs}
\begin{figure}[b!]
    \centering
    \begin{subfigure}{0.35\linewidth}
        \includegraphics[width=\textwidth]{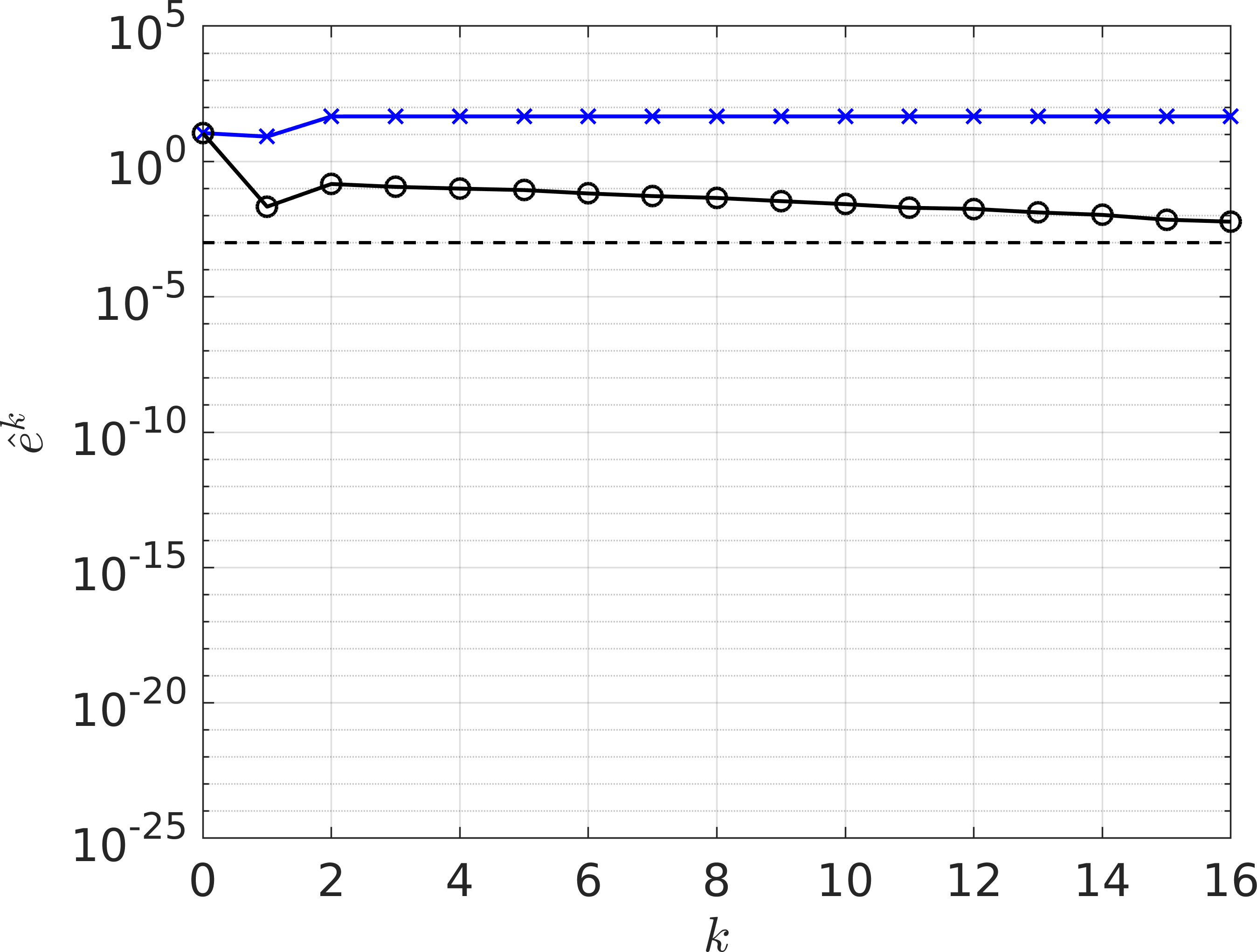}
        \caption{$q=1$}
    \end{subfigure}
    \begin{subfigure}{0.30\linewidth}
        \includegraphics[width=\textwidth]{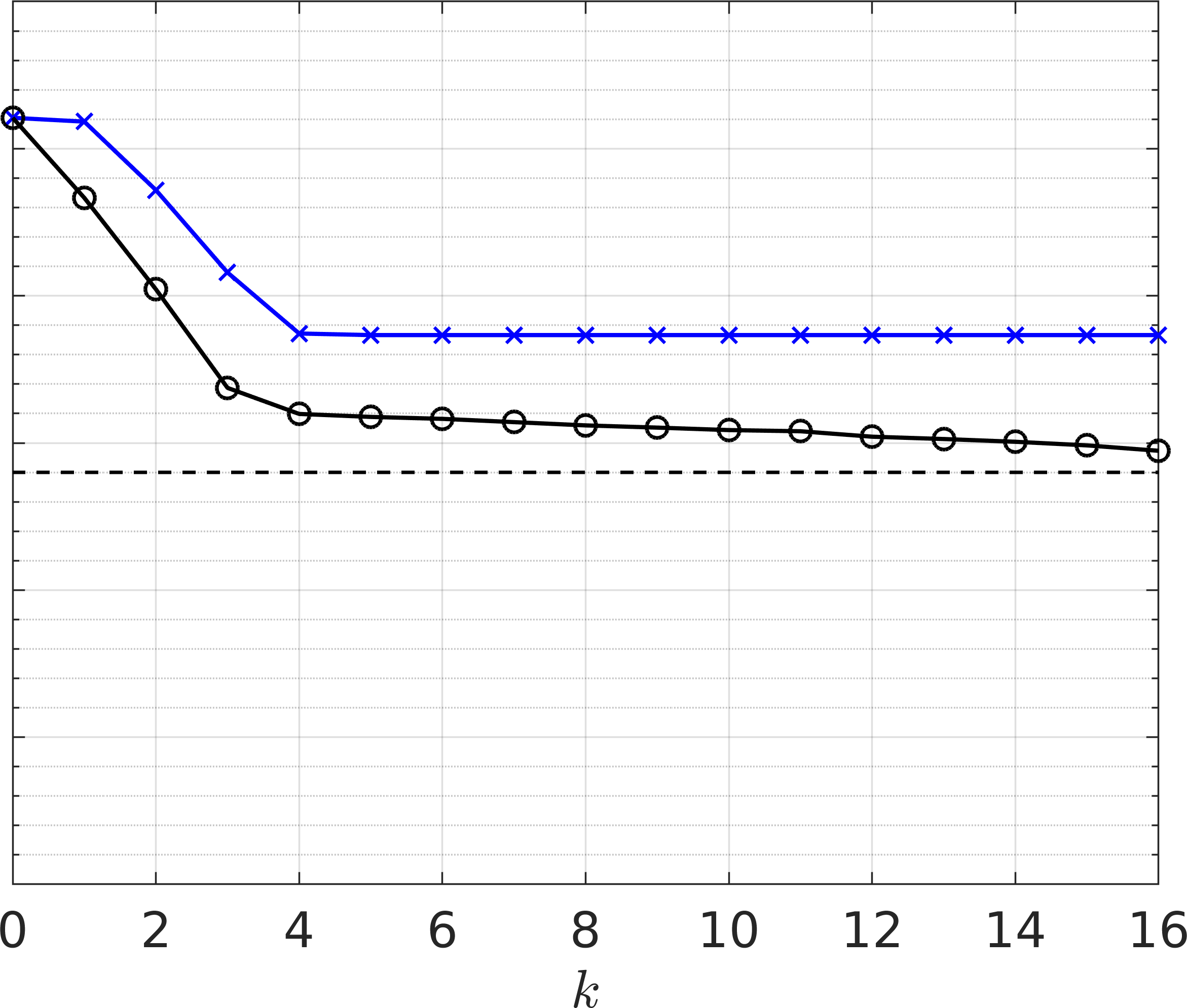}
        \caption{$q=5$}
    \end{subfigure}
    \begin{subfigure}{0.30\linewidth}
        \includegraphics[width=\textwidth]{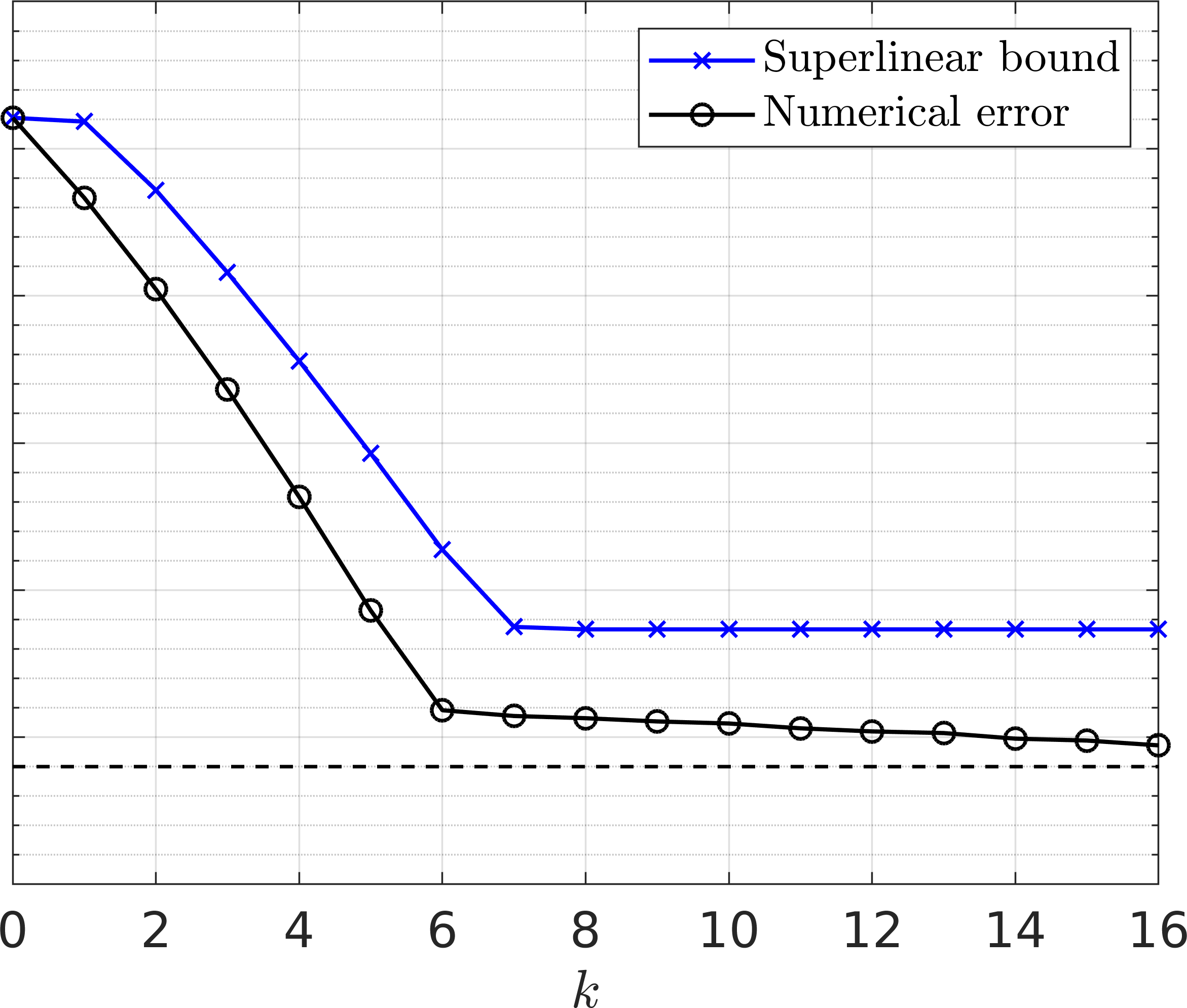}
        \caption{$q=10$}
    \end{subfigure}
    \caption{Theoretical bounds vs. numerical errors for SParareal applied to the nonlinear scalar ODE \eqref{eq:scalar_nonlinearODE} (with $B\geq1$) using state-independent Gaussian perturbations \eqref{eq:Gaussians}.
    The superlinear bound (\cref{thm:superlinear_bound}) is given in blue, the numerical error in black, and $\Delta T^{2q+1}$ in dashed black. 
    Each plot corresponds to a different level of Gaussian noise: (a) $q=1$, (b) $q=5$, and (c) $q=10$.
    Numerical errors were calculated by averaging over $500$ independent realisations of SParareal.}
    \label{fig:scalar_nonlinear}
\end{figure}
\begin{figure}[t!]
    \centering
    \begin{subfigure}{0.48\linewidth}
        \includegraphics[width=\textwidth]{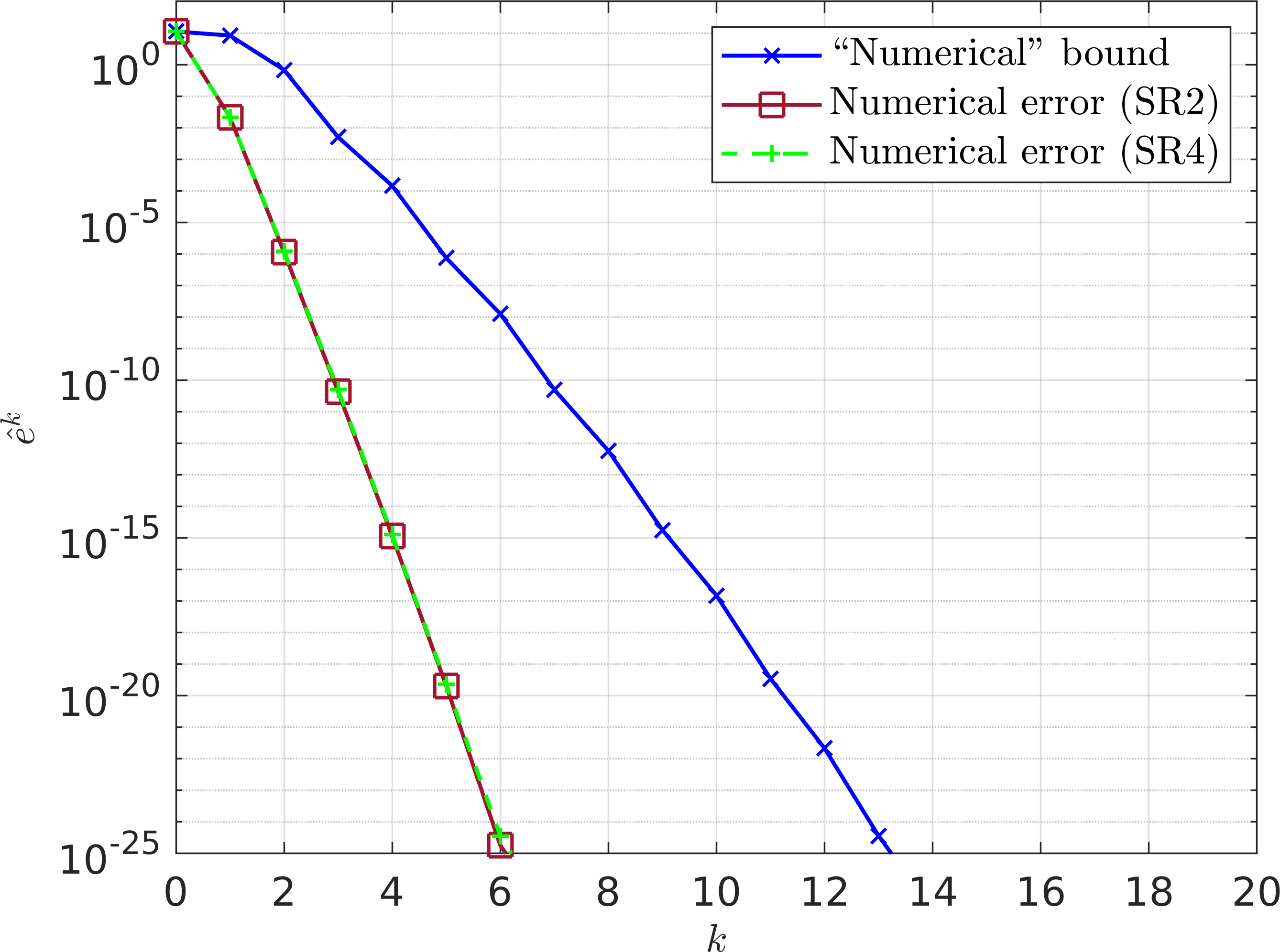}
        \caption{Sampling rules 2 and 4}
    \end{subfigure}
    \begin{subfigure}{0.49\linewidth}
        \includegraphics[width=\textwidth]{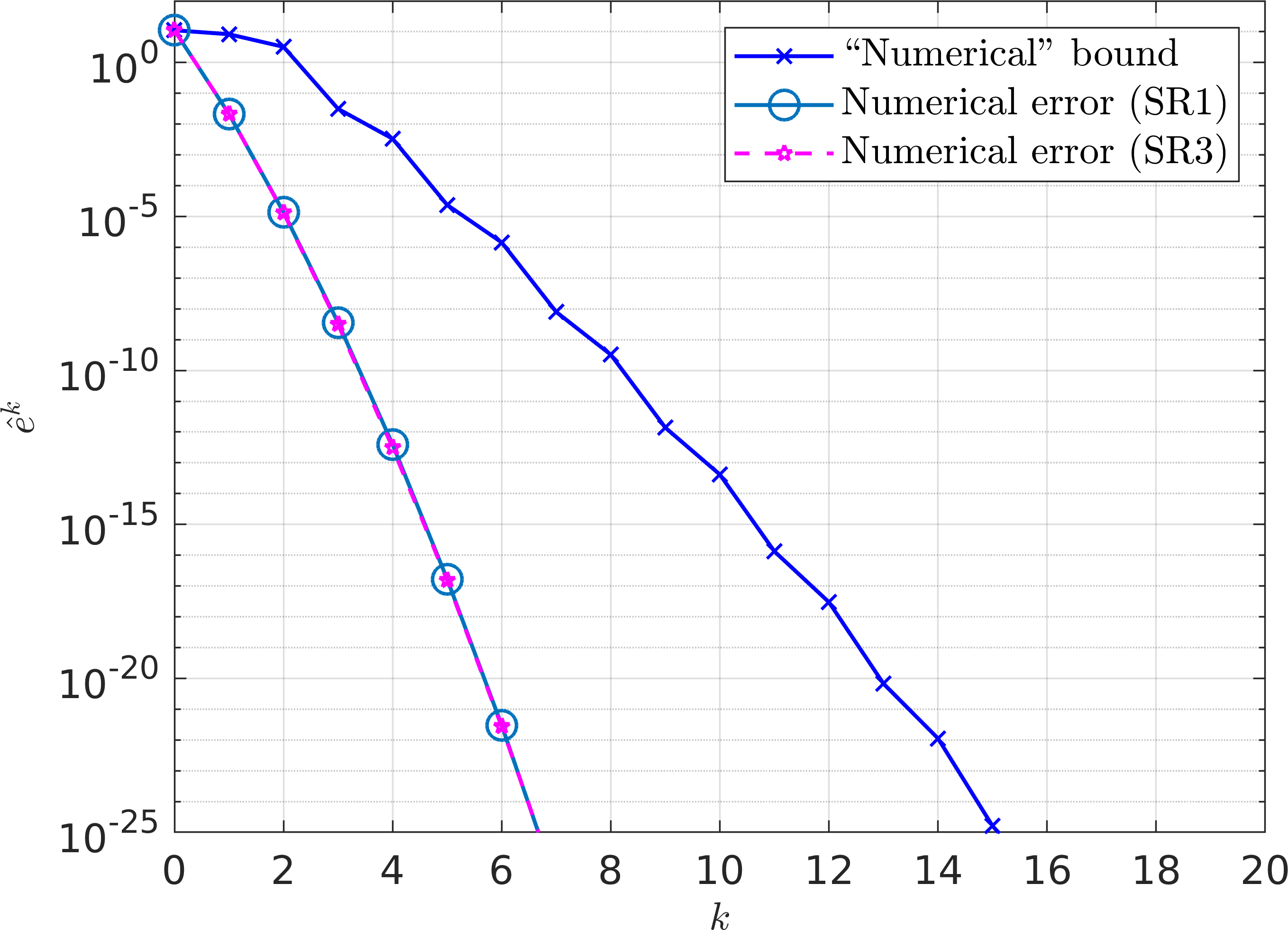}
        \caption{Sampling rules 1 and 3}
    \end{subfigure}
    \caption{``Numerical'' bounds vs. numerical errors for SParareal applied to the nonlinear scalar ODE \eqref{eq:scalar_nonlinearODE} (with $B\geq1$) using the state-dependent sampling rules.
    (a) The numerically solved recursion \eqref{eq:5-term1} is shown in blue and the numerical errors for sampling rules 2 and 4 in brown and green, respectively. 
    (b) The numerically solved recursion \eqref{eq:5-term2} is shown in blue and the numerical errors for sampling rules 1 and 3 in light blue and purple, respectively.
    Numerical errors were calculated by averaging over $500$ independent realisations of SParareal.}
    \label{fig:nonlinear_conv_SR}
\end{figure}
\begin{figure}[b!]
    \centering
    \includegraphics[width=0.49\linewidth]{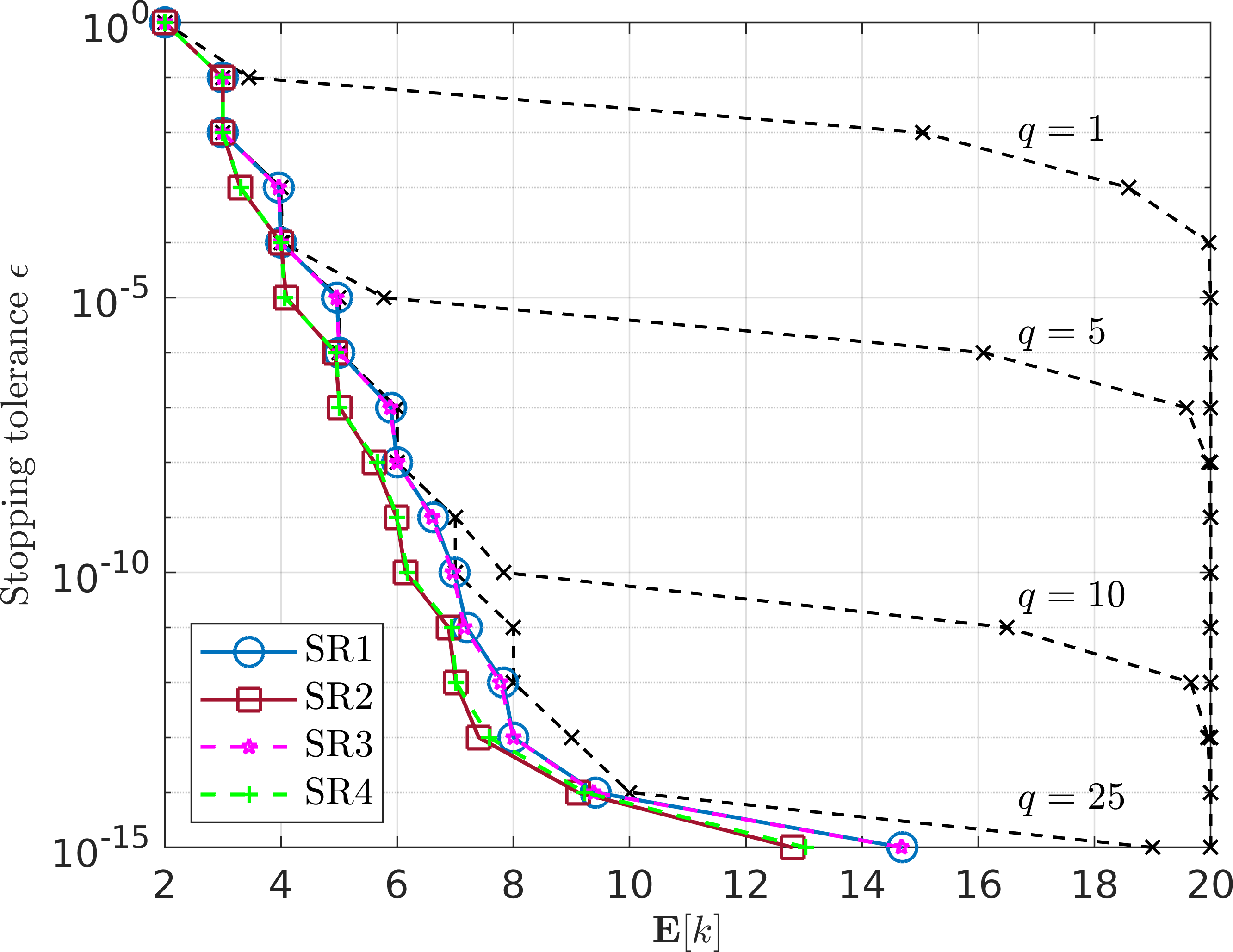}
    \caption{Expected number of iterations $k$ taken to reach stopping tolerance $\varepsilon$ \eqref{eq:tolerance} for SParareal applied to the nonlinear scalar ODE \eqref{eq:scalar_nonlinearODE} (with $B\geq1$).
    Results plotted using SParareal with each sampling rule (see legend) and the Gaussian perturbations \eqref{eq:Gaussians} for $q \in \{0,5,10,25 \}$ (dashed black lines).
    $\mathbb{E}[k]$ calculated by averaging $k$ over $500$ independent realisations of SParareal.}
    \label{fig:scalar_nonlin_epsilon_vs_k}
\end{figure}
In the following experiments, we solve the scalar nonlinear equation
\begin{equation} \label{eq:scalar_nonlinearODE} 
    \frac{\rd u}{\rd t} = \sqrt{u^2 + 2} \quad \text{over} \quad t \in [-1,1], \quad \text{with} \quad u(-1) = 5.
\end{equation}
This equation has exact solution $u(t) = \sqrt{2} \sinh(t + 1 + \sinh^{-1}(5/2))$.
We solve \eqref{eq:scalar_nonlinearODE} using SParareal with $N=20$ time slices (thus $\Delta T = 0.1$), exact solver $\F(u) = \sqrt{2} \sinh(\Delta T + \sinh^{-1}(u/\sqrt{2}))$, and forward Euler $\G = u + \Delta T \sqrt{u^2 + 2}$. 

\cref{fig:scalar_nonlinear} illustrates a good match between the superlinear bound (\cref{thm:superlinear_bound}, $B \geq 1$) and the numerical errors when using SParareal and the Gaussian perturbations \eqref{eq:Gaussians}.
One can see that at $k=0$ the error is quite large, $\mathcal{O}(10^1)$, and so even when using the forward Euler method for $\G$, the SParareal error decreases rapidly (for sufficiently large $q$).   
Given the bounds in \cref{corr:SR_24} and \cref{corr:SR_13} only hold when $B<1$, we again solve the respective recursions \eqref{eq:5-term1} and \eqref{eq:5-term2} numerically, obtaining a good match between theory and numerics when using the sampling rules (see \cref{fig:nonlinear_conv_SR}).
\cref{fig:scalar_nonlin_epsilon_vs_k} illustrates the performance of the state-independent perturbations and the sampling rules for varying stopping tolerance.
As it did for the system of linear ODEs, using SParareal with state-dependent perturbations is more effective than with the state-independent perturbations, regardless of the value of chosen value of $q$ for the Gaussian perturbations (recall that parareal can be recovered when choosing $q\geq25$).

\section{Conclusions} \label{sec:conclusions}
The SParareal algorithm solves IVPs  by perturbing solutions from the classic (deterministic) parareal scheme using a single sample drawn from a pre-specified probability distribution.
This probabilistic time-parallel scheme generates stochastic solutions to the IVP.
In this paper, we analyse the error of these stochastic numerical solutions by deriving mean-square error bounds for SParareal, equipped with different types of perturbations, applied to systems of nonlinear ODEs.

In \cref{sec:convergence}, we make assumptions about the fine ($\F$) and coarse ($\G$) numerical integrators used by SParareal, namely that $\F$ returns the exact solution to the ODE and that $\G$ has uniform local truncation error and satisfies a Lipschitz condition.
Error bounds were then derived for two types of random perturbation, one in which the random variables do not depend on the current state of the system (state-independent) and one in which they do (state-dependent).
In the state-independent case, where specific upper bounds were assumed on the second moments of the random variables, we derived both superlinear (\cref{thm:superlinear_bound}) and linear (\cref{corr:linear_bound}) bounds on the mean-square error.
In the state-dependent case, where a number different perturbations were defined according to ``sampling rules'' (from the the original work), we derived linear bounds on the errors (see \cref{corr:SR_24,corr:SR_13}).

In \cref{sec:numerics}, we illustrate these bounds, comparing them to the errors generated by running SParareal numerically. 
We demonstrate a good match between the theoretical bounds and numerical errors for systems of linear ODEs (\cref{sec:linear_ODEs}) and a scalar nonlinear ODE (\cref{sec:nonlinear_ODEs}).
Using the state-independent perturbations, we observed tight superlinear and linear bounds with respect to the numerical errors. 
However, because these perturbations do not adapt with iteration $k$ and time step $n$, their practical usage faced limitations.  
They encoded a hard lower bound on solution accuracy (of the order of the size of the second moments, see \cref{rem:loose_bound}) and more iterations were typically required to reach stopping tolerance for larger perturbations. 
Instead, the sampling rules, shown to adapt with both $k$ and $n$, did not suffer from these issues, as was previously discussed in \cite{pentland2022a}.
The derivation of the linear bounds for the sampling rules did, however, require multiple applications of the Peter-Paul inequality, resulting in less tight bounds compared to those found in the state-independent case.
Tighter bounds were observed by solving the double recursions \eqref{eq:5-term1} and \eqref{eq:5-term2} numerically.
In addition, these linear bounds required the constant $B$ to be less than one (restricting its use to problems where the Lipschitz constant for $\G$ is smaller than one) and that $\F$ be Lipschitz continuous for sampling rules 1 and 3 (an additional restriction).
In the future, it would be interesting to see if these restrictions can be avoided or whether one can derive bounds without the Lipschitz assumptions.

As high performance computing technology advances, the demand for faster and more accurate time-parallel integration methods will increase. 
With SParareal, we have seen that introducing ``local'' perturbations into an existing time-parallel scheme can enable convergence in fewer iterations (using the sampling rules) and can, on average, result in higher accuracy solutions (refer to numerical experiments in \cite{pentland2022a}).
Following multiple realisations of SParareal, these solutions can form a measure of uncertainty, i.e. a ``global'' distribution, over the true solution, the accuracy of which can be estimated by the error bounds derived in this paper. 
Further work is required to investigate whether similar bounds can be derived for the originally presented SParareal scheme (where $M$ samples can be taken from the probability distributions instead of just one) which is able to locate solutions with increasing accuracy and numerical speedup when increasing numbers of samples are taken.
In addition, in most practical applications, the exact flow map $\F$ is unknown and so it would be advantageous to investigate what happens when one relaxes this assumption, taking $\F$ to be a numerical flow map.


\section*{Acknowledgements}
The authors would like to thank Han Cheng Lie for helpful discussions.
KP is funded by the Engineering and Physical Sciences Research Council through the MathSys~II CDT (EP/S022244/1) as well as the Culham Centre for Fusion Energy.
This work has partly been carried out within the framework of the EUROfusion Consortium and has received funding from the Euratom research and training programme 2014--2018 and 2019--2020 under grant agreement No.~633053.
The views and opinions expressed herein do not necessarily reflect those of any of the above-named institutions or funding agencies.


\bibliographystyle{abbrvnat}  
\bibliography{references}  

\begin{thebibliography}{28}
\providecommand{\natexlab}[1]{#1}
\providecommand{\url}[1]{\texttt{#1}}
\expandafter\ifx\csname urlstyle\endcsname\relax
  \providecommand{\doi}[1]{doi: #1}\else
  \providecommand{\doi}{doi: \begingroup \urlstyle{rm}\Url}\fi

\bibitem[Bal(2005)]{bal2005}
G.~Bal.
\newblock On the convergence and the stability of the parareal algorithm to
  solve partial differential equations.
\newblock \emph{Lecture Notes in Computational Science and Engineering},
  40:\penalty0 425--432, 2005.
\newblock \doi{10.1007/3-540-26825-1\_43}.

\bibitem[Bal(2006)]{bal2006}
G.~Bal.
\newblock Parallelization in time of (stochastic) ordinary differential
  equations, 2006.
\newblock Pre-print:
  \href{https://www.stat.uchicago.edu/~guillaumebal/PAPERS/paralleltime.pdf}{www.stat.uchicago.edu/guillaumebal}.

\bibitem[Bal and Maday(2002)]{bal2002}
G.~Bal and Y.~Maday.
\newblock A ``parareal'' time discretization for non-linear {PDE}s with
  application to the pricing of an american put.
\newblock In \emph{Recent Developments in Domain Decomposition Methods}, pages
  189--202. Springer, Berlin, Heidelberg, 2002.
\newblock \doi{10.1007/978-3-642-56118-4\_12}.

\bibitem[Bal and Wu(2008)]{bal2008}
G.~Bal and Q.~Wu.
\newblock Symplectic parareal.
\newblock In \emph{Lecture Notes in Computational Science and Engineering},
  volume~60, pages 401--408. Springer, Berlin, Heidelberg, 2008.
\newblock \doi{10.1007/978-3-540-75199-1\_51}.

\bibitem[Carrel et~al.(2022)Carrel, Gander, and Vandereycken]{carrel2022}
B.~Carrel, M.~J. Gander, and B.~Vandereycken.
\newblock Low-rank parareal: a low-rank parallel-in-time integrator, 2022.
\newblock \arXiv{2203.08455}.

\bibitem[Conrad et~al.(2017)Conrad, Girolami, S{\"a}rkk{\"a}, Stuart, and
  Zygalakis]{conrad2017}
P.~R. Conrad, M.~Girolami, S.~S{\"a}rkk{\"a}, A.~Stuart, and K.~Zygalakis.
\newblock Statistical analysis of differential equations: introducing
  probability measures on numerical solutions.
\newblock \emph{Statistics and Computing}, 27\penalty0 (4):\penalty0
  1065--1082, 2017.
\newblock \doi{10.1007/s11222-016-9671-0}.

\bibitem[Dai et~al.(2013)Dai, Le~Bris, Legoll, and Maday]{dai2013}
X.~Dai, C.~Le~Bris, F.~Legoll, and Y.~Maday.
\newblock Symmetric parareal algorithms for hamiltonian systems.
\newblock \emph{{ESAIM}: Mathematical Modelling and Numerical Analysis},
  47\penalty0 (3):\penalty0 717--742, 2013.
\newblock \doi{10.1051/m2an/2012046}.

\bibitem[Elwasif et~al.(2011)Elwasif, Foley, Bernholdt, Berry, Samaddar,
  Newman, and Sanchez]{elwasif2011}
W.~R. Elwasif, S.~S. Foley, D.~E. Bernholdt, L.~A. Berry, D.~Samaddar, D.~E.
  Newman, and R.~Sanchez.
\newblock A dependency-driven formulation of parareal: Parallel-in-time
  solution of {PDEs} as a many-task application.
\newblock In \emph{{{MTAGS}}'11 - {{Proceedings}} of the 2011 {{ACM
  International Workshop}} on {{Many Task Computing}} on {{Grids}} and
  {{Supercomputers}}, {{Co}}-Located with {{SC}}'11}, pages 15--24, New York,
  NY, 2011. ACM Press.
\newblock \doi{10.1145/2132876.2132883}.

\bibitem[Engblom(2009)]{engblom2009}
S.~Engblom.
\newblock Parallel in time simulation of multiscale stochastic chemical
  kinetics.
\newblock \emph{Multiscale Model. Simul.}, 8:\penalty0 46--68, 2009.
\newblock \doi{10.1137/080733723}.

\bibitem[Gander and Petcu(2008)]{gander2008krylov}
M.~Gander and M.~Petcu.
\newblock Analysis of a krylov subspace enhanced parareal algorithm for linear
  problems.
\newblock \emph{{ESAIM}: Proceedings}, 25:\penalty0 114--129, 2008.
\newblock \doi{10.1051/proc:082508}.

\bibitem[Gander(2015)]{gander2015}
M.~J. Gander.
\newblock 50 {Years} of {Time Parallel Time Integration}.
\newblock In \emph{Multiple {Shooting} and {Time Domain Decomposition
  Methods}}, pages 69--113. Springer, 2015.
\newblock \doi{10.1007/978-3-319-23321-5\_3}.

\bibitem[Gander and Hairer(2008)]{gander2008}
M.~J. Gander and E.~Hairer.
\newblock Nonlinear convergence analysis for the parareal algorithm.
\newblock In \emph{Lecture {Notes} in {Computational Science} and
  {Engineering}}, volume~60, pages 45--56. Springer, 2008.
\newblock \doi{10.1007/978-3-540-75199-1\_4}.

\bibitem[Gander and Vandewalle(2007)]{gander2007}
M.~J. Gander and S.~Vandewalle.
\newblock Analysis of the parareal time-parallel time-integration method.
\newblock \emph{SIAM J. Sci. Comput.}, 29:\penalty0 556--578, 2007.
\newblock \doi{10.1137/05064607X}.

\bibitem[Gander et~al.(2022)Gander, Lunet, Ruprecht, and Speck]{gander2022}
M.~J. Gander, T.~Lunet, D.~Ruprecht, and R.~Speck.
\newblock A unified analysis framework for iterative parallel-in-time
  algorithms, 2022.
\newblock \arXiv{2203.16069}.

\bibitem[Grigori et~al.(2021)Grigori, Hirstoaga, Nguyen, and
  Salomon]{grigori2021}
L.~Grigori, S.~A. Hirstoaga, V.-T. Nguyen, and J.~Salomon.
\newblock Reduced model-based parareal simulations of oscillatory singularly
  perturbed ordinary differential equations.
\newblock \emph{Journal of Computational Physics}, 436, 2021.
\newblock ISSN 0021-9991.
\newblock \doi{10.1016/j.jcp.2021.110282}.

\bibitem[Hairer et~al.(1993)Hairer, N{\o}rsett, and Wanner]{hairer1993}
E.~Hairer, S.~P. N{\o}rsett, and G.~Wanner.
\newblock \emph{Solving {Ordinary Differential Equations I}: {Nonstiff}
  {Problems}}.
\newblock Springer Series in Computational Mathematics. Springer-Verlag, second
  edition, 1993.
\newblock \doi{10.1007/978-3-540-78862-1}.

\bibitem[Legoll et~al.(2020)Legoll, Leli{\`e}vre, Myerscough, and
  Samaey]{legoll2020}
F.~Legoll, T.~Leli{\`e}vre, K.~Myerscough, and G.~Samaey.
\newblock Parareal computation of stochastic differential equations with
  time-scale separation: A numerical convergence study.
\newblock \emph{Comput. Vis. Sci.}, 23:\penalty0 1--18, 2020.
\newblock \doi{10.1007/s00791-020-00329-y}.

\bibitem[Legoll et~al.(2022)Legoll, Leli{\`e}vre, and Sharma]{legoll2022}
F.~Legoll, T.~Leli{\`e}vre, and U.~Sharma.
\newblock An adaptive parareal algorithm: Application to the simulation of
  molecular dynamics trajectories.
\newblock \emph{SIAM J. Sci. Comput.}, 44\penalty0 (1):\penalty0 B146--B176,
  2022.
\newblock ISSN 1064-8275.
\newblock \doi{10.1137/21M1412979}.

\bibitem[Lie et~al.(2019)Lie, Stuart, and Sullivan]{lie2019}
H.~C. Lie, A.~M. Stuart, and T.~J. Sullivan.
\newblock Strong convergence rates of probabilistic integrators for ordinary
  differential equations.
\newblock \emph{Statistics and Computing}, 29:\penalty0 1265--1283, 2019.
\newblock \doi{10.1007/s11222-019-09898-6}.

\bibitem[Lie et~al.(2022)Lie, Stahn, and Sullivan]{lie2022}
H.~C. Lie, M.~Stahn, and T.~J. Sullivan.
\newblock Randomised one-step time integration methods for deterministic
  operator differential equations.
\newblock \emph{Calcolo}, 59\penalty0 (1):\penalty0 13, 2022.
\newblock \doi{10.1007/s10092-022-00457-6}.

\bibitem[Lions et~al.(2001)Lions, Maday, and Turinici]{lions2001}
J.~L. Lions, Y.~Maday, and G.~Turinici.
\newblock R\'esolution d'{{EDP}} par un sch\'ema en temps
  {\scriptsize$\mathord{\ll}$}parar\'eel{\scriptsize$\mathord{\gg}$}.
\newblock \emph{Comptes Rendus Acad. Sci. Ser. I Math.}, 332\penalty0
  (7):\penalty0 661--668, 2001.
\newblock \doi{10.1016/S0764-4442(00)01793-6}.

\bibitem[Maday and Mula(2020)]{maday2020}
Y.~Maday and O.~Mula.
\newblock An adaptive parareal algorithm.
\newblock \emph{J. Comput. Appl. Math.}, 377:\penalty0 112915--112915, 2020.
\newblock \doi{10.1016/j.cam.2020.112915}.

\bibitem[Nievergelt(1964)]{Nievergelt1964}
J.~Nievergelt.
\newblock {Parallel methods for integrating ordinary differential equations}.
\newblock \emph{Communications of the ACM}, 7:\penalty0 731--733, 1964.
\newblock \doi{10.1145/355588.365137}.

\bibitem[Ong and Schroder(2020)]{ong2020}
B.~W. Ong and J.~B. Schroder.
\newblock Applications of time parallelization.
\newblock \emph{Comput. Vis. Sci.}, 23, 2020.
\newblock \doi{10.1007/s00791-020-00331-4}.

\bibitem[Pentland et~al.(2022{\natexlab{a}})Pentland, Tamborrino, Samaddar, and
  Appel]{pentland2022a}
K.~Pentland, M.~Tamborrino, D.~Samaddar, and L.~C. Appel.
\newblock Stochastic parareal: An application of probabilistic methods to
  time-parallelization.
\newblock \emph{{SIAM} J. Sci. Comput.}, pages S82--S102, 2022{\natexlab{a}}.
\newblock ISSN 1064-8275.
\newblock \doi{10.1137/21M1414231}.

\bibitem[Pentland et~al.(2022{\natexlab{b}})Pentland, Tamborrino, Sullivan,
  Buchanan, and Appel]{pentland2022b}
K.~Pentland, M.~Tamborrino, T.~J. Sullivan, J.~Buchanan, and L.~C. Appel.
\newblock {GParareal}: A time-parallel {ODE} solver using {Gaussian} process
  emulation, 2022{\natexlab{b}}.
\newblock \arXiv{2201.13418}.

\bibitem[Samaddar et~al.(2010)Samaddar, Newman, and S{\'a}nchez]{samaddar2010}
D.~Samaddar, D.~E. Newman, and R.~S{\'a}nchez.
\newblock Parallelization in time of numerical simulations of fully-developed
  plasma turbulence using the parareal algorithm.
\newblock \emph{J. Comput. Phys.}, 229:\penalty0 6558--6573, 2010.
\newblock \doi{10.1016/j.jcp.2010.05.012}.

\bibitem[Samaddar et~al.(2019)Samaddar, Coster, Bonnin, Berry, Elwasif, and
  Batchelor]{samaddar2019}
D.~Samaddar, D.~P. Coster, X.~Bonnin, L.~A. Berry, W.~R. Elwasif, and D.~B.
  Batchelor.
\newblock Application of the parareal algorithm to simulations of {ELMs} in
  {ITER} plasma.
\newblock \emph{Comput. Phys. Commun.}, 235:\penalty0 246--257, 2019.
\newblock \doi{10.1016/j.cpc.2018.08.007}.

\end{thebibliography}


\newpage
\begin{appendices} \label{appendix}
\crefalias{section}{appendix}

\section{Standard results} \label{app:technical}
Here we state some results that we make repeated use of.

\begin{lemma}[Peter-Paul Inequality] \label{lem:youngs}
For any $\bu, \bv \in \Reals^d$ and $\delta > 0$, we have that
\begin{align} \label{eq:youngs2}
    2 \| \bu \| \| \bv \| \leq \delta \| \bu \|^2 + \delta^{-1} \| \bv \|^2.
\end{align}
\end{lemma}

\begin{theorem}[Binomial Theorem] \label{thm:binomial}
For $|x|<1$ and some $m \in \mathbb{N}$, we have that
\begin{align} \label{eq:binom}
    \frac{1}{(1-x)^m} = \sum_{i=0}^{\infty} \binom{i+m-1}{i} x^i. 
\end{align}
\end{theorem}

\section{Generating Function Method} \label{app:GFM}
Here, we solve two recurrence relations using generating functions.
These two variable (``double'') recurrences crop up often in convergence analysis involving parareal (and other PinT) algorithms and have been used in a number of settings---refer to \cite{gander2008}, \cite{carrel2022}, and \cite{gander2022} for examples.
\begin{lemma} \label{lem:recursion}
Let $e^k_n$ be a non-negative sequence and $A, B, D, \Lambda \in \Reals$ be non-negative constants. If $e^k_n$ satisfies
\begin{align} \label{eq:rec1}
    e^{k+1}_{n+1} \leq A e^k_n + B e^{k+1}_n + \Lambda, \quad e^1_{n+1} \leq D + B e^1_n,
\end{align}
for $2 \leq k < n \leq N$ and $e^k_0 = 0$ $\forall k\geq0$, then 
\begin{align*}
e^k_n \leq D A^{k-1} \sum_{\ell = 0}^{n-k} \binom{\ell+k-1}{\ell} B^{\ell} + \Lambda \sum_{j=0}^{k-2} \sum_{\ell=0}^{n-(j+1)} \binom{\ell+j}{\ell} A^j B^{\ell}.
\end{align*}
\end{lemma}
\begin{proof}
For $k \geq 1$, define the generating function for $e^k_n$ as
\begin{align} \label{eq:gen}
    g_k(x) = \sum_{n=1}^{\infty} e^k_n x^n.
\end{align}
Multiply \eqref{eq:rec1} by $x^{n+1}$ and sum from $n$ equals zero to infinity to obtain
\begin{align*}
    \sum_{n=0}^{\infty} e^{k+1}_{n+1} x^{n+1} &\leq A\sum_{n=0}^{\infty} e^{k}_{n} x^{n+1} + B\sum_{n=0}^{\infty} e^{k+1}_{n} x^{n+1} + \Lambda \sum_{n=0}^{\infty} x^{n+1}, \\
    \sum_{n=0}^{\infty} e^{1}_{n+1} x^{n+1} &\leq D \sum_{n=0}^{\infty} x^{n+1} + B\sum_{n=0}^{\infty} e^{1}_{n} x^{n+1}.
\end{align*}
Using \eqref{eq:gen}, binomial theorem \eqref{thm:binomial}, and recalling that $e^k_0=0 \ \forall k\geq0$, we can write these expressions as
\begin{align*}
    g_{k+1}(x) \leq \frac{Ax}{1-Bx} g_k(x) + \frac{\Lambda x}{(1-x)(1-Bx)}, \quad g_1 (x) \leq \frac{Dx}{(1-x)(1-Bx)},
\end{align*}
for $|x|<1$, which can be solved iteratively to give
\begin{align*}
    g_{k}(x) \leq \Big( \frac{Ax}{1-Bx} \Big)^{k-1} \frac{Dx}{(1-x)(1-Bx)} + \frac{\Lambda x}{(1-x)(1-Bx)} \sum_{j=0}^{k-2} \Big( \frac{Ax}{1-Bx} \Big)^j.
\end{align*}

Re-arranging terms, this can be written as
\begin{align*}
    g_{k}(x) \leq D A^{k-1} x^k \Big( \frac{1}{1-Bx} \Big)^{k} \frac{1}{(1-x)} + \Lambda \sum_{j=0}^{k-2} A^j x^{j+1} \Big( \frac{1}{1-Bx} \Big)^{j+1} \frac{1}{1-x}.
\end{align*}
The first term can be expressed as
\begin{align*}
    D A^{k-1} x^k \Big( \frac{1}{1-Bx} \Big)^{k} \frac{1}{(1-x)} 
    &= D A^{k-1} x^k \Bigg( \sum_{i=0}^{\infty} \binom{i+k-1}{i} (Bx)^i \Bigg) \Bigg( \sum_{i=0}^{\infty} x^i \Bigg) \\
    &= D A^{k-1} x^k \sum_{m=0}^{\infty} \Bigg( \sum_{\ell = 0}^{m} \binom{\ell+k-1}{\ell} B^{\ell} \Bigg) x^m \\
    &= D A^{k-1} \sum_{n=k}^{\infty} \Bigg( \sum_{\ell = 0}^{n-k} \binom{\ell+k-1}{\ell} B^{\ell} \Bigg) x^n.
\end{align*}
The first line follows by applying \eqref{eq:binom} twice, the second line using the Cauchy product, and the third by setting $n = m + k$.
The second term can be expressed as 
\begin{align*}
    \Lambda \sum_{j=0}^{k-2} A^j x^{j+1} \Big( \frac{1}{1-Bx} \Big)^{j+1} \frac{1}{1-x} 
    &= \Lambda \sum_{j=0}^{k-2} A^j x^{j+1} \Bigg( \sum_{i=0}^{\infty} \binom{i+j}{i} (Bx)^i \Bigg) \Bigg( \sum_{i=0}^{\infty} x^i \Bigg) \\
    &= \Lambda \sum_{j=0}^{k-2} A^j x^{j+1} \sum_{m=0}^{\infty} \Bigg( \sum_{\ell = 0}^{m} \binom{\ell+j}{\ell} B^{\ell} \Bigg) x^m \\
    &= \Lambda \sum_{n=j+1}^{\infty} \Bigg( \sum_{j=0}^{k-2} \sum_{\ell=0}^{n-(j+1)} \binom{\ell+j}{\ell} A^j B^{\ell} \Bigg) x^n.
\end{align*}
These steps follows as they did for the first term, except that we now set $n = m+j+1$ instead of $n = m + k$ in the last step.
Combining these expressions we get
\begin{align*}
    g_{k}(x) = \sum_{n=1}^{\infty} e^k_n x^n \leq D A^{k-1} \sum_{n=k}^{\infty} \Bigg( \sum_{\ell = 0}^{n-k} \binom{\ell+k-1}{\ell} B^{\ell} \Bigg) x^n + \Lambda \sum_{n=j+1}^{\infty} \Bigg( \sum_{j=0}^{k-2} \sum_{\ell=0}^{n-(j+1)} \binom{\ell+j}{\ell} A^j B^{\ell} \Bigg) x^n.
\end{align*}
Equating the coefficients in $x^n$ on both sides of the inequality we obtain the bound. 
\end{proof}

The initial condition for the recursion \eqref{eq:rec1} can be written differently depending on the available information, i.e.\ one could instead use $e^1_{n+1} \leq D:= \hat{e}^1$, slightly altering the final bound obtained.

\begin{lemma} \label{lem:recursion2}
Let $\hat{e}^k$ be a non-negative sequence and $\tilde{A}, \tilde{B} \in \Reals$ be non-negative constants. If $\hat{e}^k$ satisfies
\begin{align} \label{eq:rec2}
    \hat{e}^{k+1} \leq \tilde{A} \hat{e}^k + \tilde{B} \hat{e}^{k-1},
\end{align}
with initial conditions $\hat{e}^0$ and $\hat{e}^1$, then
\begin{align*}
    \hat{e}^k \leq \hat{e}^0 \biggr[ \frac{\tilde{A} + \sqrt{\tilde{A}^2 + 4 \tilde{B}}}{2} \biggr]^k.
\end{align*}
\end{lemma}
\begin{proof}
Define the following generating function for $\hat{e}^k$:
\begin{align*}
    g(x) = \sum_{k=0}^{\infty} \hat{e}^k x^k.
\end{align*}
Multiply \eqref{eq:rec2} by $x^{k+1}$ and sum from $k$ equals one to infinity to obtain
\begin{align*}
    \sum_{k=1}^{\infty} \hat{e}^{k+1} x^{k+1} &\leq \tilde{A} \sum_{k=1}^{\infty} \hat{e}^{k} x^{k+1} + \tilde{B} \sum_{k=1}^{\infty} \hat{e}^{k-1} x^{k+1}.
\end{align*}
Shifting indices, rearranging, and using the initial conditions we get
\begin{align*}
    g(x) = \sum_{k=0}^{\infty} \hat{e}^k x^k \leq \frac{\hat{e}^0(1 - \tilde{A}x) + \hat{e}^1 x}{1 - \tilde{A}x - \tilde{B}x^2}.
\end{align*}
Expanding the right hand side in powers of $x^k$, the coefficients give us 
\begin{align*}
    \hat{e}^k \leq \frac{1}{2\sqrt{\tilde{A}^2 + 4\tilde{B}}} \bigg[ (\tilde{A} \hat{e}^0 + \hat{e}^0 \sqrt{\tilde{A}^2 + 4\tilde{B}} - 2\hat{e}^1) \lambda_1^k + (-\tilde{A} \hat{e}^0 + \hat{e}^0 \sqrt{\tilde{A}^2 + 4\tilde{B}} + 2\hat{e}^1) \lambda_2^k \bigg],
\end{align*}
where
\begin{align*}
    \lambda_{1,2} = \frac{\tilde{A} \pm \sqrt{\tilde{A}^2 + 4 \tilde{B}}}{2}.
\end{align*}
Without loss of generality, we use that $\lambda_1 \geq \lambda_2$ to simplify the bound and obtain
\begin{align*}
    \hat{e}^k \leq \hat{e}^0 \lambda_1^k,
\end{align*}
which yields the desired result.
\end{proof}

\end{appendices}


\end{document}